\documentclass[a4paper,10pt]{article}   
\usepackage{graphicx}
\usepackage{caption}
\usepackage{subcaption} 
\usepackage{float}   
\usepackage{fancybox}		  
\usepackage{makeidx} 
\usepackage{verbatim}
\usepackage{listings} 
\usepackage{fancyvrb}
\usepackage{amsfonts}
\usepackage{color}
\usepackage{colortbl}
\usepackage{amsmath}
\usepackage{color}
\usepackage{dsfont}
\usepackage{epsfig}
\usepackage{verbatim}
\usepackage{listings}

\usepackage{amssymb}
\usepackage{amsbsy}
\usepackage{amsmath}
\usepackage{enumerate}
\usepackage{stmaryrd}
\usepackage{mathtools}
\usepackage{tikz}
\usepackage{bbm}

\usepackage{fancyhdr}

\usepackage{amsthm}

\DeclarePairedDelimiter\Lnorm{\|}{\|_{L^1(\R)}}
\DeclarePairedDelimiter\Lnormt{\|}{\|_{L^1([0,T]\times\R)}}
\DeclarePairedDelimiter\Ldnorm{\|}{\|_{L^2(\R)}}
\DeclarePairedDelimiter\Linorm{\|}{\|_{L^\infty(\R)}}
\DeclarePairedDelimiter\Linormt{\|}{\|_{L^\infty([0,T]\times\R)}}

\DeclarePairedDelimiter\Lnormp{\|}{\|_{L^1(\R^+)}}

\DeclarePairedDelimiter\Linormp{\|}{\|_{L^\infty(\R^+)}}
\DeclarePairedDelimiter\Linormtp{\|}{\|_{L^\infty([0,T]\times\R^+)}}

\newtheorem{prop}{Proposition}
\newtheorem{theorem}{Theorem}
\newtheorem{lemma}{Lemma}
\newtheorem{rem}{Remark}[section]
\newtheorem{definition}{Definition}

\usepackage[margin=1in]{geometry}

\def\N {\mathbb{N}}

\def\R {\mathbb{R}}

\newcommand{\dfdt}{\frac{\partial f}{\partial t}}

\newcommand{\fin}{f^{\mathrm{in}}}

\newcommand{\ddt}{\frac{\partial }{\partial t}}

\newcommand{\Ddt}{\frac{\mathrm{d}}{\mathrm{d}t}}

\newcommand{\ie}{\textit{i.e.} }
\newcommand{\tf}{\tilde{f}}
\newcommand{\fp}{f^+}

\newcommand{\fe}{f_\varepsilon}
\newcommand{\fep}{\fe^+}
\newcommand{\fem}{\fe^-}

\newcommand{\indm}{\mathds{1}_{x\leq\varepsilon}}
\newcommand{\indp}{\mathds{1}_{x > \varepsilon}}

\newcommand{\supp}{\mathrm{supp}}

\newcommand{\bb}{\bar\beta(t)}

\newcommand{\dx}{\,\mathrm{d}x}
\newcommand{\dt}{\,\mathrm{d}t}
\newcommand{\ds}{\,\mathrm{d}s}

\begin{document}

\title{Kinetic approach to the collective dynamics of the rock-paper-scissors binary game}

\author{Nastassia Pouradier Duteil\thanks{Sorbonne Universit\'e, Inria, Universit\'e Paris-Diderot SPC, CNRS, Laboratoire Jacques-Louis Lions, Paris, France} 
\ and 
Francesco Salvarani\thanks{\textcolor{black}{L\'eonard de Vinci P\^ole Universitaire, Research Center, 92916 Paris La D\'efense, France} \& Universit\`a degli Studi di Pavia, Dipartimento di
Matematica, I-27100 Pavia, Italy}
}

\date{}
\maketitle

\bibliographystyle{abbrv}

\begin{abstract}
\textcolor{black}{This article studies the kinetic dynamics of the {\it rock-paper-scissors} binary game. We first prove existence and uniqueness of the solution of the kinetic equation and subsequently we prove the rigorous derivation of the quasi-invariant limit for two meaningful choices of the domain of definition of the independent variables. We notice that the domain of definition of the problem plays a crucial role and heavily influences the behavior of the solution.  
The rigorous proof of the relaxation limit does not need the use of entropy estimates for ensuring compactness.}
\end{abstract}

\section{Introduction}

The binary zero-sum game {\it rock-paper-scissors} provides a simple framework for any two-player contest where each player has an equal probability of winning, losing or tying.
\textcolor{black}{It is one of the sub-classes of three-strategy games, and it has been generalised in many ways (for example, punishment games and reward games), both in the static and in the evolutionary contexts. We refer to \cite{game2, game1} and their bibliographies for a more complete description of the subject.} 

The extension of this game to the scale of an entire population has been the subject of several studies in bacterial ecology and evolution 
\cite{2012PhRvE}. In particular, it has been related to the promotion of strain diversity for populations of Escherichia coli \cite{2002Natur}
and other bacteria \cite{2004Natur,Patk2012}, as well as to the stabilization of bacteria populations.
In evolutionary game dynamics, 
it has allowed to describe the cyclic competition of species in ecosystems (for instance in species extinction and coexistence \cite{Shi2010BasinsOA}, or in male reproductive strategies \cite{SL96}). 
These applications use the simple concept of the game to model cyclic competition between species, where the first species dominates the second one, the second dominates the third, and the third dominates the first. Pattern formation can be shown between the densities of populations of each species \cite{2011PhRvE}.

In this article, we take a different approach and we introduce and study the kinetic version of the rock-paper-scissors game, in which instead of inter-species competition, each agent within the unique population can compete with all the other agents.
In the modeling of a kinetic game,
we suppose to have an infinite set of players, which form temporary pairs of players through random encounters. The two agents of a pair play the game only once, and then they search for another competitor, and so on.
We consider the simplest possible situation, namely we suppose that the population is fully interconnected and that there are no forbidden pairwise interactions. Moreover, we assume that the game is only able to modify the value of the individual exchange variable $x$
(which can have any meaningful interpretation, for example it could be intended as the wealth of the individuals, if the agents exchange a given amount of money).
It is clear that the optimal individual strategy consists in randomly playing one of the possible options of the game in a uniformly random order.
\textcolor{black}{Consequently, instead of focusing on the evolution of each player's strategy, as in standard evolutionary games, in this article we consider that each player adopts the same optimal strategy and we study the evolution of the population's wealth distribution.}

From the mathematical point of view, the situation is modeled through a two-level description. We first consider the dynamics of a binary encounter and then we use it for the description of the global population, represented by a density probability whose time evolution obeys to a partial integro-differential equation of kinetic type.
It is worth noticing that the mathematical form of the model heavily depends not only on the precise rules of the game, but also on the domain of the exchange variable. We study here two meaningful possible situations. In the first case $x\in\R$ whereas, in the second case, only interactions that give as possible outcome non-negative values of the exchange variable are allowed (i.e. $x\in\R^+$). If we interpret the exchange variable as the wealth of the individuals, the second case treats situations in which debts are not allowed.
We will see that the resulting mathematical properties of the models are quite different and that the choice of the domain of definition of the exchange variable is a crucial feature of the system.

In particular, when $x\in\R$ the model becomes linear, and it has the form of the semi-discrete implicit Euler scheme of the heat equation. On the other hand, 
if the exchange variable is such that $x\in\R^+$, the model is nonlinear and exhibits concentration effects. 

The main questions studied in the article include the well-posedness of the two problems and their relaxation limits as the payoff $h$ of the game tends to zero in an appropriate time scaling.
This problem is very natural and it has been studied by different authors in several articles (see, for example, \cite{bou-sal2,mcnamara1992inelastic}
and the review article \cite{MR2744702}).
In the first case, when $x\in\R$, the limit of the semi-discrete equation is no other than the classical heat equation in $\R$. 
However, the limit of the second problem, when $x\in\R^+$, involves a diffusion equation with a diffusion rate depending on the mass of the diffusing quantity. \textcolor{black}{By denoting with $u=u(t,x)$ the unknown, this nonlinear and nonlocal equation has, in one space dimension, the form
$$
 \partial_t u(t,x) = \frac{\eta}{3}\left( \int_{\R^+} u(t,x_*) \dx_*\right) \partial_x^2 u(t,x),
$$
and is then supplemented with opportune initial and boundary conditions (see equation \eqref{eq:systlim}).
}
This introduces a nonlinearity as well as a non-local effect. Such kinds of non-local diffusion equations have, to our knowledge, not yet been studied in the literature, \textcolor{black}{which is mainly focused on nonlinearities of porous media type \cite{vazquez2007porous} of non-localities on fractional Laplacian type \cite{bucur2016nonlocal}.}

The two problems being structurally different, it is not surprising that the mathematical arguments in the relaxation procedure are very different. Whereas the linear model is quite straightforward (the relaxation procedure is nothing but the proof of the convergence of the solution of the semi-discrete implicit Euler scheme to the solution of the heat equation), the \textcolor{black}{nonlinear} case is more intricate, \textcolor{black}{because it requires to handle weak $L^1$ compactness in $\R$}.

In the classical framework of kinetic theory a crucial tool for handling weak $L^1$ compactness is given by the entropy estimate, which guarantees the uniform integrability of the density through the De La Vall\'ee-Poussin criterion \cite{golse2005hydrodynamic}. However, because of the positiveness requirement for the post-interaction exchange variables, the existence of a Lyapunov functional for the kinetic model studied in this article is not straightforward when $x\in\R^+$. For this reason, we have solved the relaxation limit by means of a strategy which does not require the existence of a Lyapunov functional for the whole problem. It consists in separating the concentrating part of the solution of the kinetic model from the non-concentrating one and then by treating them in a separate way (see Section \ref{Sec4}).
Up to our knowledge, it is the first rigorous relaxation limit for a \textcolor{black}{nonlinear} kinetic model which does not use entropy estimates for proving compactness.
Moreover, the structure of the diffusion equation satisfied by the limit being interesting by itself, we hope that our work will motivate further studies on the target equation.
 
The structure of the article is as follows. After a general introduction to the model, in Section \ref{Sec3} we treat the linear case, namely when $x\in\R$, whereas Section \ref{Sec4} considers the \textcolor{black}{nonlinear} model for $x\in\R^+$. \textcolor{black}{Lastly}, Section \ref{Sec5} provides an introduction to the study of the nonlinear diffusion equation obtained as the limit of the nonlinear equation studied in Section \ref{Sec4}.
In the whole article, the theoretical analysis is supplemented with the corresponding numerical simulations.

\section{The mathematical description of the problem}

As usual in the kinetic approach, our model is formulated by looking at two different levels. The first one describes the binary interaction dynamics and the second one the time evolution of the population density.

Consider a population which interacts pairwise by playing a game in which each player has three choices: $(0,1,2)$. At each game, the winner's exchange variable increases by $h$ and the loser's decreases by $h$. If there is no winner, both players' exchange variables remain unchanged. The wins and losses are determined \textcolor{black}{in Table \ref{table}}. Each player wins and loses with probability $1/3$.

\begin{table}[h!]
\centering
\begin{tabular}{c|c|c|c|}
& 0 & 1 & 2 \\[4pt]
\hline
\ & \ & \ & \\[-6pt]
$0 \quad$ & $(0,0)$ & $(h,-h)$ & $(-h,h)$ \\[4pt] 
\hline
\ & \ & \ & \\[-6pt]
$1 \quad$  &  $(-h,h)$ & $(0,0)$ & $(h,-h)$ \\[4pt] 
\hline
\ & \ & \ & \\[-6pt]
$2 \quad$  &  $(h,-h)$ & $(-h,h)$ & $(0,0)$ \\[4pt]
\end{tabular}
\caption{\textcolor{black}{Payoff table of the rock-paper-scissors game}}
\label{table}
\end{table}

Let $(x,x_*)\in\R^2$ denote the  values of the exchange variables of two agents after the interaction, and $(x', x'_*)\in\R^2$ their values just before the interaction. If the player with post-interaction exchange variable $x$ wins, $x$ and $x_*$ satisfy: 
\begin{equation}
\begin{cases}
x = x' + h\\
x_*= x'_* - h.
\end{cases}
\end{equation}\label{eq:winloss}
If the players play the same, there is no winner and we have
\begin{equation}
\begin{cases}
x = x'\\
x_* = x'_*.
\end{cases}
\end{equation}\label{eq:draw}

At the collective level, the system of interacting individuals is described by a distribution function $f=f(t,x)$ defined on $[0,T]\times\R$, where
$t\in [0,T]$ is the time variable and $x\in X \subseteq \R$ the exchange variable.
For each subset $D\subseteq X$, the integral
$$
\int_D f(t,x) \dx
$$
represents the number of individuals for which the exchange variable belongs to $D$.

It is clear that the mathematical structure of the model heavily depends on the set $X$. In this article we will focus our attention on two paradigmatic situations: in the first one, we do not impose any constraints on the value of the exchange variables, i.e. $X=\R$, whereas in the second model we suppose that
$X=\R^+$ (if the exchange variable represents the agent's wealth, it means that debts are not allowed).

Throughout the article, $h\in\R^+$ denotes the wager in play and $\eta$ is the probability that two agents interact. 

It is clear that the best strategy, in the sense of game theory, consists in playing the three possible options in a random order, with the same probability. In what follows, we will build the kinetic models corresponding to the two different domains of the exchange variable considered above.

\section{The unconstrained model}
\label{Sec3}

In this section we suppose that $x\in\R$. We link the time derivative of the density function with the possible situations giving, as outcome, the desired result, weighted with the probability that the corresponding situation occurs.
The model takes the following form:
\begin{equation}
\begin{cases}
\displaystyle
\frac{1}{\eta}\dfdt(t,x) = \frac{1}{3}\int_\R f(t,x-h) f(t,x_*) \dx_* + \frac{1}{3}\int_\R f(t,x+h) f(t,x_*) \dx_* -\frac{2}{3}\int_\R f(t,x) f(t,x_*) \dx_* \\[13pt]
f(0,x) = \fin(x)
\end{cases}
\end{equation}
for all $ (t,x)\in\R^+\times\R$.
If we suppose that the density $f$ is non-negative, it is easy to verify that the mass of the population is formally conserved, and we denote it by 
$$
\int_{\R}f(t,x)\dx=\Lnorm{f(t,\cdot)}=\Lnorm{\fin}:=\rho \quad \text{ for all } t\in [0,T].
$$
Similarly, the first moment is conserved, so that we have
$$
\int_{\R}xf(t,x)\dx=\int_\R x\fin(x) \dx  \quad \text{ for all } t\in [0,T].
$$
Hence, the evolution of the density $f$ can be rewritten as: 
\begin{equation}
\label{eq:dyn}
\begin{cases}
\displaystyle
\frac{1}{\eta}\dfdt(t,x) = \frac{1}{3}\rho\left[ f(t,x+h) + f(t,x-h)  -2 f(t,x)\right] \\[13pt]
f(0,\cdot) = \fin(\cdot)
\end{cases}
\end{equation}
for all $ (t,x)\in\R^+\times\R$.

Equation \eqref{eq:dyn} can be seen as a semi-discrete (in space) version of the heat equation.

\subsection{Basic properties}

From the point of view of modeling, it is clear that the natural space of the density $f$ is the space of positive measures. However, from the mathematical point of view, it is convenient to work in a bigger space, namely the space of tempered distributions, which will allows us to use some Fourier transform techniques.
We hence define the solution as follows:

\begin{definition}\label{def:sol_discrete_heat}
Denote with $\langle\langle\, \cdot \, , \, \cdot \, \rangle\rangle$ the duality $(\mathcal{S}'(\R\times\R), \mathcal{S}(\R\times\R))$ and with
  $\langle\, \cdot \, , \, \cdot \, \rangle$ the duality $(\mathcal{S}'(\R), \mathcal{S}(\R))$.
 A solution of \eqref{eq:dyn} in the sense of tempered distributions is an element $f\in\mathcal{S}'(\R\times\R)$ such that
$\supp(f)\subset \R^+\times\R$ and such that, for all $\varphi\in\mathcal{S}(\R\times\R)$, we have
$$
\frac{3}{\rho\eta}\langle\langle f, \varphi_t \rangle\rangle 
 + \frac{3}{\rho\eta} \langle \fin , \varphi(0,\, \cdot\, ) \rangle +
\langle\langle f, \left[\varphi(t,x+h)+\varphi(t,x-h)-2 \varphi(t,x)\right] \rangle\rangle =0.
$$
\end{definition}

The following proposition holds:
\begin{prop}\label{prop:tempered}
Let $\fin\in \mathcal{S}'(\R)$. Then, there exists a unique solution of problem \eqref{eq:dyn} in the sense of tempered distributions. If, moreover, 
$\fin\in L^2(\R)$, then the unique solution \textcolor{black}{satisfies} $f\in C^\infty([0,T],L^2(\R))$. If, in addition, $\fin\geq 0$ for a.e. $x\in \R$, then $f(t,x)\geq 0$ for a.e.
$(t,x)\in\R^+\times \R$.
If, furthermore, $\fin\in L^p(\R)\cap L^2(\R)$, then $f\in C^\infty([0,T],L^p(\R)\cap L^2(\R))$ for all $p\in[1,+\infty]$. In particular, for $p=1$, the mass conservation is guaranteed.
\end{prop}

\begin{proof}
The Fourier transform being an isomorphism of $\mathcal{S}'$ onto itself, we can work with the Fourier transform of the problem.
Let $\hat f$ be the partial Fourier transform of $f$ with respect to the $x$ variable, and let $\xi$ be the associated Fourier variable. 
Hence, the tempered distribution $\hat f$ satisfies, for all $\varphi\in\mathcal{S}(\R\times\R)$,
$$
\langle\langle f, \hat\varphi_t \rangle\rangle 
 +  \langle \fin , \hat\varphi(0,\, \cdot\, ) \rangle + \frac{2\rho\eta}{3}
\left\langle\left\langle f,  \left[\cos\left({h\, \xi} \right )-1\right] \hat \varphi \right\rangle\right\rangle =0,
$$
which is the distributional formulation of the problem
\begin{equation}
\label{eq:dyn_Fourier}
\begin{cases}
\displaystyle
\partial_t  \hat f =  \frac{2\rho\eta}{3}\left [\cos\left({h\,\xi}\right )-1\right] \hat f
\\[13pt]
\hat f(0,\cdot) = \hat\fin(\cdot).
\end{cases}
\end{equation}

In order to prove uniqueness, let us suppose, by contradiction, that there are two solutions $f_1$ and $f_2$ to \eqref{eq:dyn_Fourier}. Hence, the difference
$u:=f_1-f_2$ satisfies, for all $\varphi\in\mathcal{S}(\R\times\R)$
$$
\left\langle\!\!\left\langle u,  
\hat\varphi_t+ \frac{2\rho\eta}{3}
\left[\cos\left({h\, \xi}\right )-1\right] \hat \varphi \right\rangle\!\!\right\rangle =0.
$$
Since $\supp(u)\subset \R^+\times\R$, then the distribution $u \left [\cos\left({h\,\xi}\right )-1\right]$ has support in $ \R^+\times\R$, and hence we can limit
ourselves to test-functions $ \varphi\in\mathcal{S}(\R\times\R)$ \textcolor{black}{such that} $\supp(\varphi)\subset \R^+\times\R$.

Consider now a function $\varphi\in\mathcal{S}(\R\times\R)$ with $\supp(\varphi)\subset (-T,T)\times\R$. The Schwartz class
$\mathcal{S}(\R\times\R)$ being invariant under the Fourier transform, the uniqueness of the
solution of \eqref{eq:dyn_Fourier} is guaranteed after proving that, for all $\psi\in\mathcal{S}(\R\times\R)$ with $\supp(\psi)\subset (-T,T)\times\R$, there exists a function $\hat \varphi\in\mathcal{S}(\R\times\R)$ such
that the problem
$$
\hat \varphi_t+ \frac{2\rho\eta}{3} \left[\cos\left({h\, \xi}\right )-1\right] \hat \varphi =\hat\psi \qquad  \hat\varphi(t,x)\vert_{t=T}=0
$$
has a solution. The structure of the problem allows us to obtain the explicit form of $\hat \varphi$, namely
$$
\hat \varphi(t,\xi)= \int_0^T \hat\psi(t',\xi )\exp\left\{\frac{2\rho\eta}{3} \left[\cos\left({h\, \xi}\right )-1\right] (t'-t)\right\}\dt',
$$
which shows that the requirements about the support and the regularity of $\hat \varphi$ are fulfilled. Hence,
from $\langle\langle u, \hat\psi\rangle\rangle=0$ for all $\psi\in\mathcal{S}(\R\times\R)$ with $\supp(\psi)\subset (-T,T)\times\R$, we deduce that
$\hat u=0$ in $\mathcal{S}'(\R\times\R)$ and consequently that $u=0$ in $\mathcal{S}'(\R\times\R)$. The uniqueness is hence proved.

By linearity, it is easy to find the explicit formulation of the (unique) solution of \eqref{eq:dyn_Fourier} in $\mathcal{S}'$:
$$
\langle\langle \hat f, \varphi \rangle\rangle =  \left\langle \hat\fin(\cdot) , \varphi(t,\cdot) \exp \left [\frac{2\rho\eta t}{3} \left(\cos\left({h\,\cdot} \right )-1\right)\right]
\right\rangle.
$$
The Fourier transform being an isomorphism of $\mathcal{S}'$ onto itself, consequently, there exists a unique solution of \eqref{eq:dyn} in $\mathcal{S}'$. 
If, moreover, $\fin\in L^2(\R)$, we can deduce from the explicit expression of the solution in the Fourier representation, that there exists a unique solution $f\in C([0,T],L^2(\R))$ by Plancherel's theorem, and \eqref{eq:dyn_Fourier} can be intended in the standard $L^2$-sense. Moreover, by a bootstrap argument, one easily proves that $f\in C^\infty([0,T],L^2(\R))$.

For any $g\in L^\infty(\R^+\times\R)$, denote its partial inverse Fourier transform by
$$
\mathcal{F}^{-1}(g)(t,x)= \frac 1{2\pi}\int_\R g(t,\xi) e^{i x \xi}\, \mathrm{d} \xi
$$
and consider
$$
F(t,x):= \mathcal{F}^{-1}
\left(  \exp\left[
\frac{2\rho\eta}{3} \left(\cos({h\, \xi} )-1\right)t \right]
\right).
$$

Hence
$$
F(t,x)=  \frac 1{2\pi} \exp \left[-\frac{2\rho\eta t}{3}  \right] \int_\R \exp\left[\frac{2\rho\eta}{3} \cos\left({h\, \xi}\right )t \right]  e^{i x \xi}\, \mathrm{d} \xi
$$
$$
= \frac 1{2\pi} \exp \left[-\frac{2\rho\eta t}{3}  \right] \int_\R \sum_{k=0}^\infty \frac 1{k!}\left[\frac{2\rho\eta}{3} \cos\left({h\, \xi}\right )t \right]^k  e^{i x \xi}\, \mathrm{d} \xi.
$$
By using the identity
$$
\cos (\alpha)= \frac{e^{i\alpha}+e^{-i\alpha}}2 \qquad \text{\textcolor{black}{for all} }\alpha\in\R,
$$
we can conclude that
\begin{equation}\label{eq:F}
\forall (t,x)\in [0,T]\times \R, \quad F(t,x) = e^{-2\rho\eta/{3}}\sum_{k=0}^{+\infty} \frac{({\eta}\rho t)^k}{3^k k!} \sum_{i=0}^k \dbinom{k}{i} \delta_0(x+(k-2i)h).
\end{equation}

Going back to \eqref{eq:dyn_Fourier}, we can finally deduce that
\begin{equation}\label{eq:Fconv}
f(t,x) = (F*\fin)(t,x) = \int_\R F(t,x-y)\fin(y)dy.
\end{equation}
The previous equation, together with the explicit form of $F$, shows that, if the initial condition $\fin$ is non-negative, then the solution is non-negative as well for all times.

Because of the regularity properties inherited by the convolution, Equation \eqref{eq:Fconv} guarantees that the solution has -- at least -- the regularity of the initial condition.
\end{proof}

Define the energy of the system $E:t\mapsto \Ldnorm{f(t,\cdot)}/2$.
The following theorem implies that -- under suitable hypotheses on the initial condition -- the energy decreases with respect to time like $t\mapsto t^{-\gamma}$ for all $\gamma<1/4$.

\begin{theorem}
Suppose that $\fin\in H^{1/2+\varepsilon}(\R)$ for some $\varepsilon>0$. Then for all $\gamma<{1}/{4}$, there exists an explicit constant $C_\gamma$ such that
\begin{equation}\label{eq:energy}
E(t)\leq \frac{C_\gamma}{t^\gamma}.
\end{equation}
\end{theorem}
\begin{proof}
The energy of the system is defined as: 
\begin{equation*}
\begin{split}
E(t) & = \frac{1}{2} \int_\R |f (t,x)|^2 dx = \frac{1}{2} \int_\R \left\vert\hat{f}(t,\xi)\right\vert^2 d\xi =
\frac{1}{2}\int_\R\left\vert\hat{\fin}(\xi)\right\vert^2 \exp \left(-{\frac{8\rho\eta}{3}\sin^2\left(\frac{h\xi}{2}\right)t}\right) d\xi \\ 
& = \frac{1}{2}\int_\R 
\left\vert\hat{\fin}(\xi)\right\vert^2\,
\left\vert \frac{8\rho\eta}{3}\sin^2\left(\frac{h\xi}{2}\right)t \right\vert^{-\gamma} 
\left|\frac{8\rho\eta}{3}\sin^2\left(\frac{h\xi}{2}\right)t\right|^\gamma \exp\left (-{\frac{8\rho\eta}{3}\sin^2\left(\frac{h\xi}{2}\right)t}\right) d\xi. 
\end{split}
\end{equation*}
Notice that for all $\gamma>0$, for all $z\geq 0$, $z^\gamma e^{-z}\leq \left({\gamma}/{e}\right)^\gamma$. Then, provided that 
$$
\xi\mapsto \left\vert\hat{\fin}(\xi)\right\vert^2 \left\vert\frac{8\rho\eta}{3}\sin^2\left(\frac{h\xi}{2}\right )t \right\vert^{-\gamma}
$$ 
is integrable on $\R\setminus \{{2k\pi}/{h}, k\in\mathbb{Z}\}$,
$$
E(t) \leq \frac{1}{2} \left(\frac{3 \gamma}{e}\right)^\gamma \sum_{k\in\mathbb{Z}}\int_{{2\pi k}/{h}}^{{2\pi(k+1)}/{h}}  \frac{\left\vert\hat{\fin}(\xi)\right\vert^2}{\left\vert{8\rho\eta}\sin^2({h\xi}/{2})t\right\vert^{\gamma}} d\xi.
$$
Since $\fin\in L^1(\R)$, $\hat\fin\in L^\infty(\R)$ so if $\gamma<{1}/{2}$, the \textcolor{black}{quantity}
$$
\int_{{2\pi k}/{h}}^{{2\pi(k+1)}/{h}} \frac{\left\vert\hat{\fin}(\xi)\right\vert^2}{\left\vert{8\rho\eta}\sin^2({h\xi}/{2})t\right\vert^{\gamma}} d\xi
$$ 
is well defined for all $k\in\mathbb{Z}$.
As we have supposed $\gamma<{1}/{4}$, we can even write, for each term of the series: 
$$
\int_{{2\pi k}/{h}}^{{2\pi(k+1)}/{h}}
 \frac{\left\vert\hat{\fin}(\xi)\right\vert^2}{\left\vert{8\rho\eta}\sin^2({h\xi}/{2})t\right\vert^{\gamma}} d\xi
\leq 
\left(\int_{{2\pi k}/{h}}^{{2\pi(k+1)}/{h}} \left\vert\hat{\fin}(\xi)\right\vert^4 d\xi\right)^{1/2}
\left( \int_{{2\pi k}/{h}}^{{2\pi(k+1)}/{h}} \frac{1}{\left\vert{8\rho\eta}\sin^2({h\xi}/{2})t\right\vert^{2\gamma}}  d\xi\right)^{1/2}.
$$
The quantity 
$$
c_\gamma:=\left( \int_{{2\pi k}/{h}}^{{2\pi(k+1)}/{h}} \frac{1}{\left\vert{8\rho\eta}\sin^2({h\xi}/{2})t\right\vert^{2\gamma}}  d\xi\right)^{1/2}
$$
is independent of $k\in \mathbb{Z}$, so 
$$
E(t) \leq \frac{1}{2} \frac{c_\gamma}{t^\gamma} \left(\frac{3\gamma}{e}\right)^\gamma \sum_{k\in\mathbb{Z}}
\left( \int_{{2\pi k}/{h}}^{{2\pi(k+1)}/{h}} 
\left\vert\hat{\fin}(\xi)\right\vert^4 d\xi\right)^{1/2},
$$
provided that this series converges. We prove that this is the case if $\fin\in H^{1/2+\varepsilon}$ for $\varepsilon>0$. Indeed, we then have $\hat\fin (1+|\xi |^2)^{1/4+\varepsilon/2}\in L^\infty(\R)$. 
Since $\hat\fin\in L^\infty(\R)$, we also have $|\hat\fin|^4 (1+|\xi|^2)^{1+2\varepsilon}\in L^\infty(\R)$. Then we have 
\begin{equation*}
\begin{split}
E(t) & \leq \frac{1}{2} \frac{c_\gamma}{t^\gamma} \left(\frac{3\gamma}{e}\right)^\gamma \sum_{k\in\mathbb{Z}}
\left(\int_{{2\pi k}/{h}}^{{2\pi(k+1)}/{h}} \frac{|\hat{\fin}(\xi)|^4 (1+|\xi|^2)^{1+2\varepsilon}}{(1+|\xi|^2)^{1+2\varepsilon}} d\xi\right)^{1/2} \\
& \leq \frac{1}{2} \frac{c_\gamma'}{t^\gamma} \left(\frac{3\gamma}{e}\right)^\gamma \sum_{k\in\mathbb{Z}}
\left(\int_{{2\pi k}/{h}}^{{2\pi(k+1)}/{h}} \frac{1}{(1+|\xi|^2)^{1+2\varepsilon}} d\xi\right)^{1/2} 
 \leq \frac{1}{2} \frac{c_\gamma'}{t^\gamma} \left(\frac{3\gamma}{e}\right)^\gamma\sqrt{\frac{2\pi}{h}} \sum_{k\in\mathbb{Z}}
\left (1+\left\vert\frac{2k\pi}{h}\right\vert^2\right )^{-\frac{1+2\varepsilon}{2}}.
\end{split}
\end{equation*}
The series on the right-hand side of the previous inequality converges, so gathering all constants under the name $C_\gamma$ we obtain
the desired inequality \eqref{eq:energy}.
\end{proof}

\subsection{The quasi-invariant limit}

As a warm-up for the relaxation limit in the case of the constrained model, whose computations will be detailed in the next section, we study the limit of the rescaled system

\begin{equation}\label{eq:dynrescaled}
\begin{cases}
\displaystyle\partial_t f_\varepsilon(t,x) = \frac{\eta}{3\varepsilon^2}\Lnorm{\fin}  \left(f_\varepsilon(t,x+\varepsilon)+f_\varepsilon(t,x-\varepsilon)-
2f_\varepsilon(t,x)\right) \\[13pt]
f(0,x) = \fin(x)
\end{cases}
\end{equation}
and show that the solution to \eqref{eq:dynrescaled} tends to the solution to the heat equation when the parameter $\varepsilon$ tends to $0$.
It is an instructive computation, because it shows that linearity, together with an $L^\infty$-bound, make quite straightforward the study of the limit in the
diffusive scaling. As we will show in the next section, the lack of these properties requires a rather different approach.

We first prove the following technical lemma, which allows us to \textcolor{black}{deduce a uniform $L^\infty$ bound and hence to work in the $L^\infty$
setting}.

\begin{lemma}\label{lemma:Linf}
Let $\fin\in  L^1(\R) \cap L^\infty(\R)$ and let $f_\varepsilon\in C^\infty([0,T];L^1(\R)\cap L^2(\R))$ denote the weak solution to the rescaled dynamics \eqref{eq:dynrescaled}. 
Then $f_\varepsilon\in L^\infty([0,T] \times\R)$ and 
$$
\Linormt{f_\varepsilon} \leq \Linorm{\fin}.
$$
\end{lemma}
\begin{proof}
The proof uses Stampacchia's truncation method \cite{S64}. 
Let $\zeta\in C^1(\R)$ be a smooth version of the positive part function, satisfying: 
 $\zeta'(s)>0$ for all $s>0$;
  $\zeta(s) = 0$ for all $s\leq 0$;
  and  $|\zeta(s)|\leq M$ for all $s\in\R$,
where $M$ is a positive constant. 
Let $\displaystyle K:= \sup_{x\in\R} \fin(x)$. 

Define $G$ as 
$$
G:t\mapsto \int_\R \int_0^{f_\varepsilon(t,x)-K} \zeta(s) \ds \; \dx.
$$
From the properties of $g$ and $f_\varepsilon$, $G\in C^1((0,T],\R)$, $G(0)=0$ and $G(t)\geq 0$ for all $t\in [0,T]$. We have
\begin{equation*}
\begin{split}
G'&(t) = \int_\R \zeta(f_\varepsilon(t,x)-K)\partial_t f_\varepsilon(t,x) \dx = \frac{\eta}{3}\rho \int_\R \zeta(f_\varepsilon(t,x)-K) ( f_\varepsilon(t,x+h) 
+ f_\varepsilon(t,x-h)  -2 f_\varepsilon(t,x)) \dx \\
& = \frac{\eta}{3}\rho \left[ \int_\R \zeta(f_\varepsilon(t,x)-K) ( f_\varepsilon(t,x+h) - f_\varepsilon(t,x) ) \dx - \int_\R \zeta(f_\varepsilon(t,x)-K) 
( f_\varepsilon(t,x) - f_\varepsilon(t,x-h) ) \dx \right] \\ 
& =   \frac{\eta\rho }{3} \int_\R \left[ \zeta\!\left(\!f_\varepsilon (t,x-\frac{h}{2})-K\!\right)\!-\zeta\!\left(\!f_\varepsilon (t,x+\frac{h}{2})-K\!\right)\!\right] \!\left[ 
\!\left(\!f_\varepsilon(t,x+\frac{h}{2})-K\!\right)\! - \!\left(\!f_\varepsilon (t,x-\frac{h}{2})-K\!\right)\!  \right]\!  \dx \\ 
&\leq 0.
\end{split}
\end{equation*}
Hence $G(t)=0$ for all $t\in [0,T]$, which implies that $f_\varepsilon(t,x)-K\leq 0$ for all $t\in [0,T]$ and \textcolor{black}{for a.e.} $x\in \R$.
\end{proof}

\begin{rem}
Lemma $\ref{lemma:Linf}$ shows that a solution initially in $L^1\cap L^\infty$ cannot blow up, and stays bounded in $L^\infty$ norm by its initial data. Equation \eqref{eq:dyn} can be seen as an semi-discretized version of the heat equation, but it does not 
have the strong smoothing effects given by the ultra-conservative estimates of the heat equation. 
\end{rem}

The main result of this subsection is gathered in the following proposition:
\begin{prop}
Let $\fin\in L^1(\R)\cap L^\infty(\R)$ and $T>0$. For each $\varepsilon>0$, let $\fe\in C([0,T];L^1(\R)\cap L^\infty(\R))$ denote the weak solution to \eqref{eq:dynrescaled}.
Then the sequence $(\fe)_{\varepsilon>0}$ converges weakly in $L^1([0,T]\times\R)$ and its limit $ \bar f := \lim_{\varepsilon\rightarrow 0} \fe$ is the unique solution to the heat equation 
\begin{equation}\label{eq:heat}
\begin{cases}
\displaystyle \partial_t  \bar f = \frac{\eta}{3}\Lnorm{\fin} \; \partial_x^2  \bar f  \\[13pt]
f(0,\cdot) = \fin.
\end{cases}
\end{equation}
\end{prop}
\begin{proof}
The rescaling of system \eqref{eq:dyn} by $\varepsilon^2$ does not change the conservation of mass, \textit{i.e.} for all $t\in [0,T]$, $\Lnorm{\fe(t,\cdot)} = \Lnorm{\fin}$.
Consequently, $\Lnormt{\fe} \leq T\Lnorm{\fin}$.
Furthermore, from Lemma \ref{lemma:Linf}, $\Linorm{\fe(t,\cdot)}\leq \Linorm{\fin}$ for all $t\in [0,T]$, so the sequence $(\fe)_{\varepsilon>0}$ is equi-integrable. 
By the Dunford-Pettis theorem, there exists a subsequence (that we denote again by $(\fe)_{\varepsilon>0}$) that converges weakly to $\bar f\in L^1([0,T]\times\R)$. 
Let us show that $\bar f$ is a weak solution to \eqref{eq:heat}.
Let $\varphi\in C_c^\infty([0,T]\times\R)$. As a weak solution to \eqref{eq:dynrescaled}, $\fe$ satisfies: 
$$
\int_0^T \int_\R \fe(t,x) \partial_t \varphi(t,x) \dx \dt +  \int_\R \fin(x) \varphi(0,x) \dx + 
$$
$$
\frac{\eta}{3}\Lnorm{\fin} \frac{1}{\varepsilon^2} \int_0^T \int_\R \varphi(t,x) \left(\fe(t,x+\varepsilon)+\fe(t,x-\varepsilon)-2\fe(t,x)\right) \dx \dt=0. 
$$
Now
\begin{equation*}
\begin{split}
& \frac{\eta}{3}\Lnorm{\fin} \frac{1}{\varepsilon^2} \int_0^T \int_\R \varphi(t,x) \left(\fe(t,x+\varepsilon)+\fe(t,x-\varepsilon)-2\fe(t,x)\right) \dx  \dt  \\
& = \frac{\eta}{3}\Lnorm{\fin} \frac{1}{\varepsilon^2} \int_0^T \int_\R \fe(t,x) \left(\varphi(t,x+\varepsilon)+\varphi(t,x-\varepsilon)-2\varphi(t,x)\right) \dx  \dt \\
& = \frac{\eta}{3}\Lnorm{\fin} \int_0^T \int_\R \fe(t,x) \left( \partial_x^2 \varphi(t,x) + \mathcal{O}(\varepsilon^2) \right) \dx \dt .
\end{split}
\end{equation*}
Since $\partial_t\varphi\in L^\infty([0,T]\times \R)$ and $\partial_x^2\varphi\in L^\infty([0,T]\times \R)$, from the weak convergence of $\fe$ to $\bar f$ we have: 
$$
\lim_{\varepsilon\rightarrow 0} \int_0^T \int_\R \fe(t,x) \partial_t \varphi(t,x) \dx \dt = 
\int_0^T \int_\R \bar f(t,x) \partial_t \varphi(t,x) \dx \dt
$$ 
as well as 
$$
\lim_{\varepsilon\rightarrow 0} \int_0^T \int_\R  \fe(t,x) \partial_x^2 \varphi(t,x) \dx \; \dt = \int_0^T \int_\R  \bar f(t,x) \partial_x^2\varphi(t,x)  \dx \dt.
$$
Hence $\bar f$ satisfies
 $$
\int_0^T \int_\R \bar f(t,x) \partial_t \varphi(t,x) \dx  \dt + \int_\R \fin(x) \varphi(0,x) \dx + \frac{\eta}{3}\Lnorm{\fin}  \int_0^T \int_\R \bar f(t,x) \partial_x^2\varphi(t,x)  \dx \dt=0
 $$
for all $\varphi\in C_c^\infty([0,T]\times\R)$.
\end{proof}

\subsection{Numerical simulations}
\label{subs/numsim1}

We collect here some simulations, computed by means of a standard finite difference solver, describing the time evolution of the solution for various values of
the payoff $h$. In all the simulations, the initial condition is $\fin\equiv \mathbbm{1}_{[0,1)}$.

It is apparent that there is no gain in regularity and that, for small $h$, the solution numerically tends to a gaussian profile, as in the case of the standard
heat equation.

\begin{figure}[h!]
\begin{center}
\includegraphics[width=0.32\textwidth]{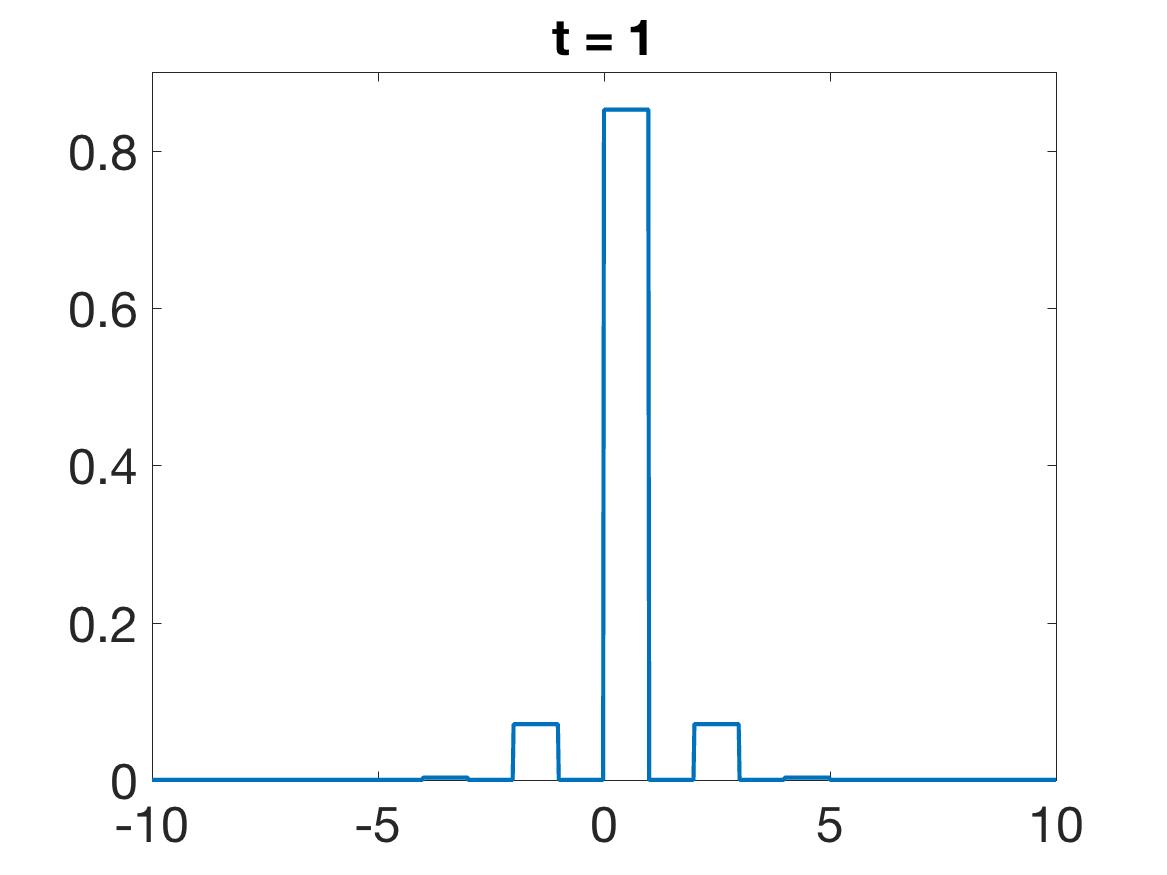}
\includegraphics[width=0.32\textwidth]{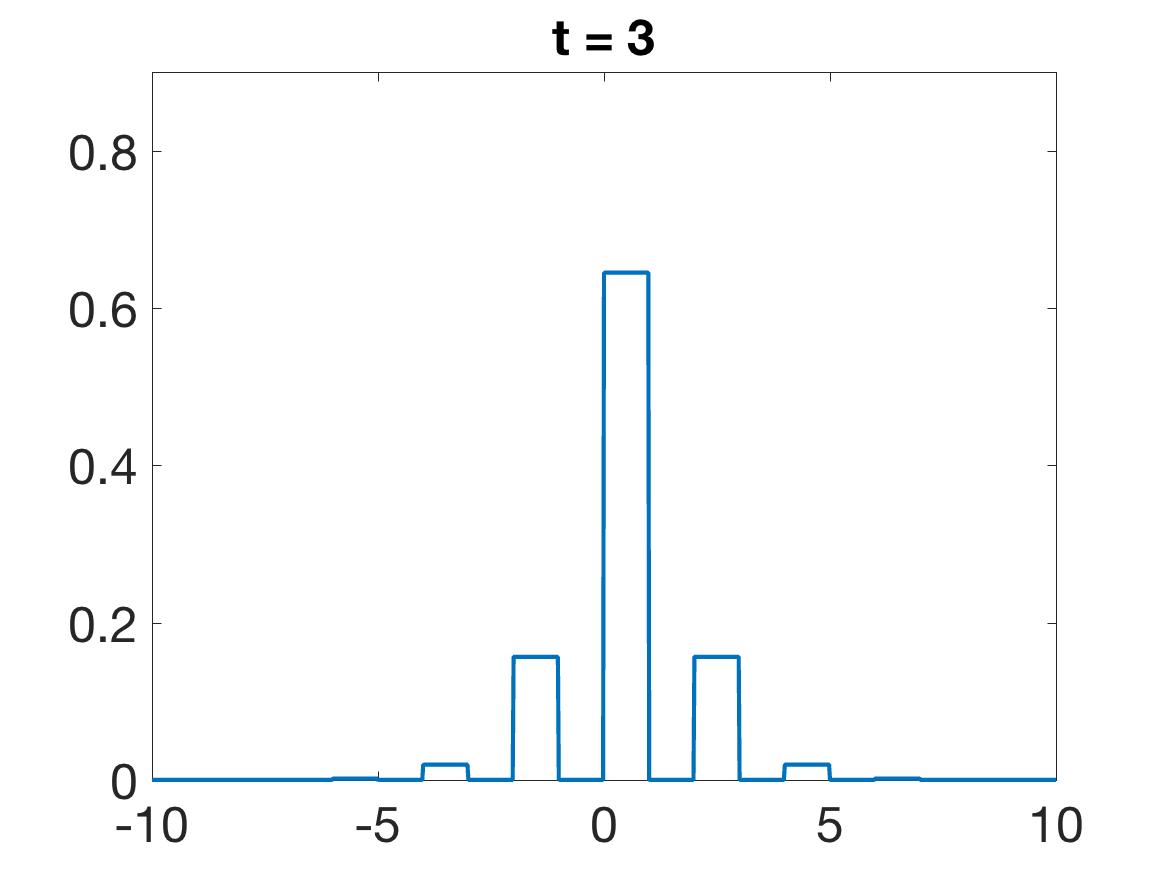}
\includegraphics[width=0.32\textwidth]{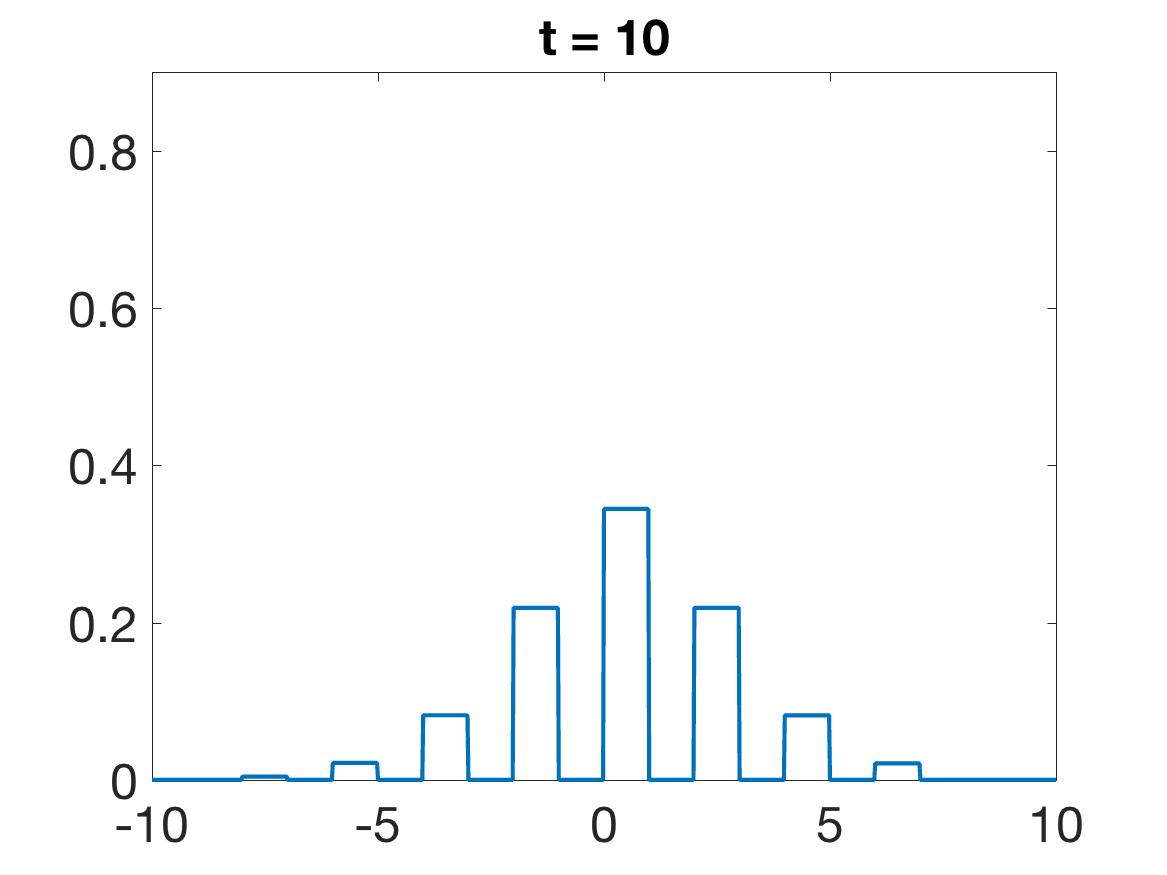}
\end{center}
\caption{Solution of the rescaled equation \eqref{eq:dynrescaled} with $\varepsilon=2$ and initial data $\fin\equiv \mathbbm{1}_{[0,1)}$.}
\label{fig1}
\end{figure}

\begin{figure}[h!]
\begin{center}
\includegraphics[width=0.32\textwidth]{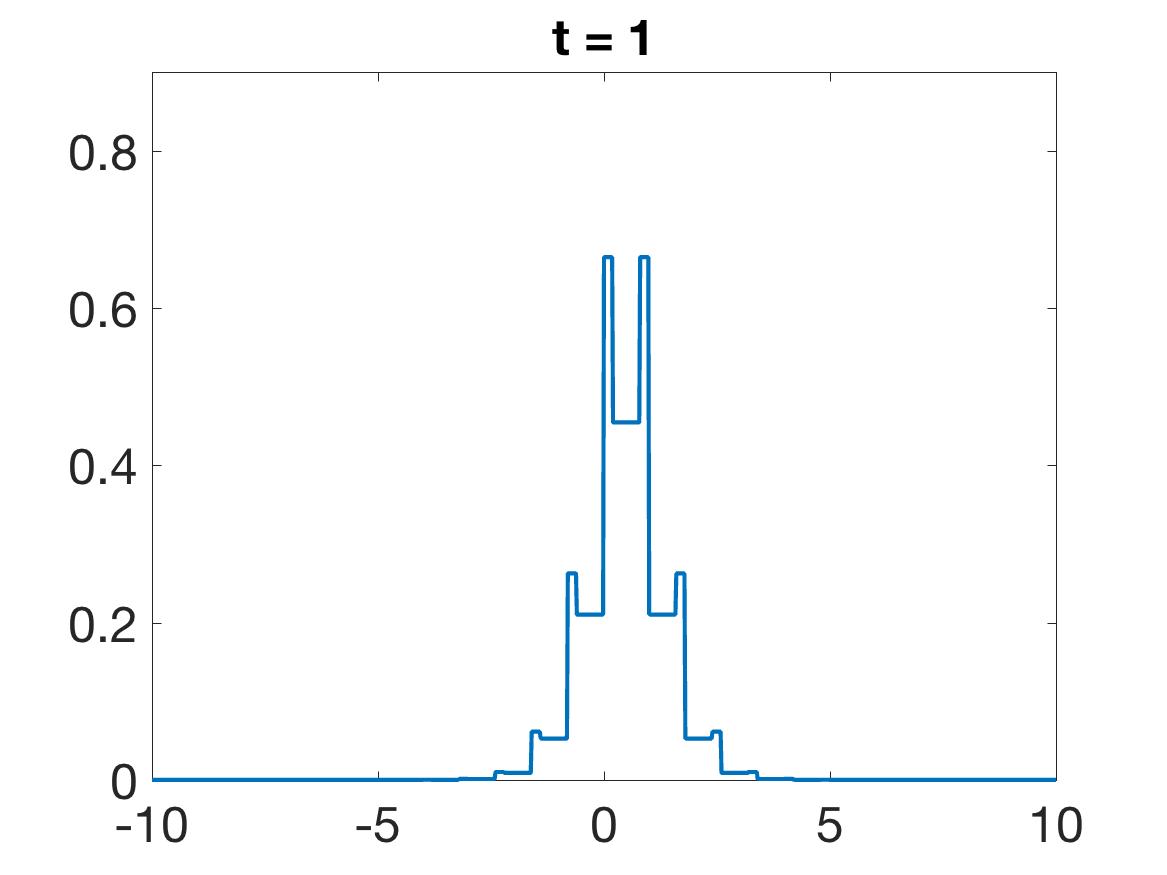}
\includegraphics[width=0.32\textwidth]{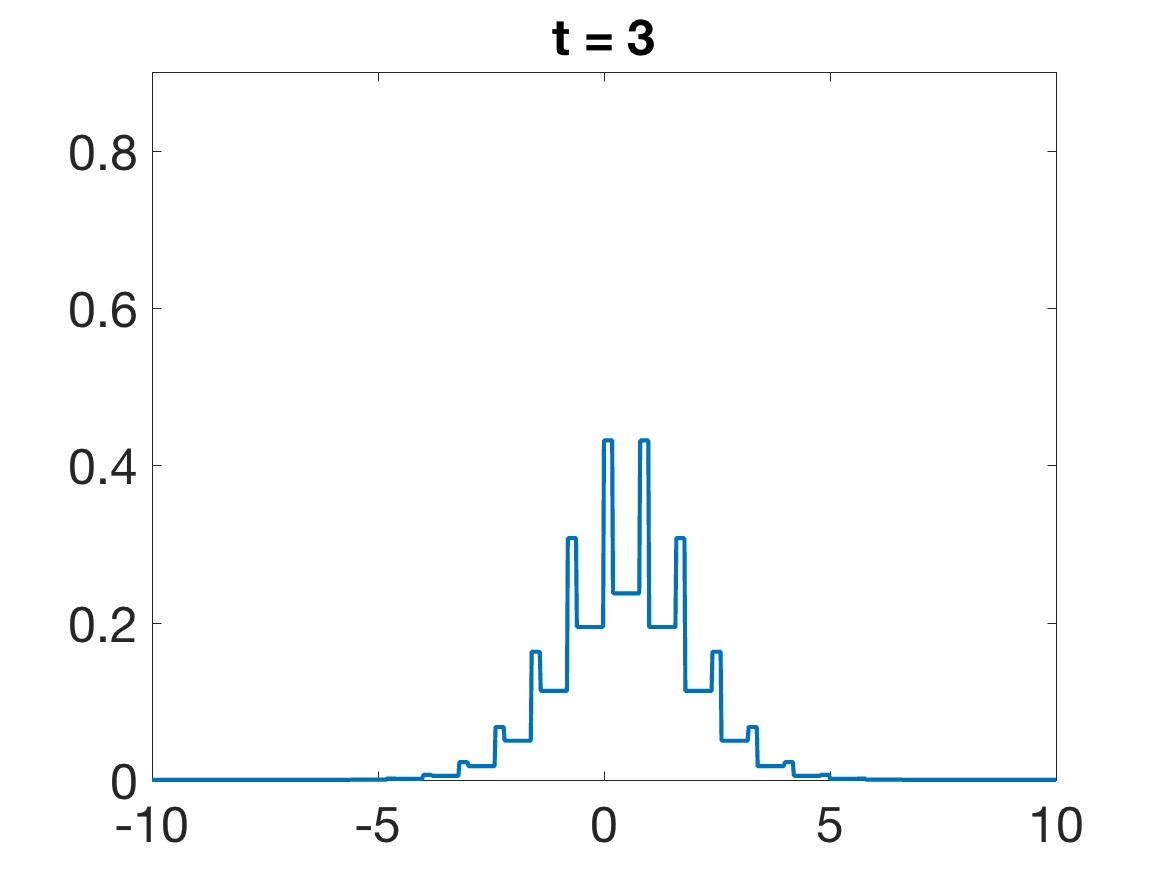}
\includegraphics[width=0.32\textwidth]{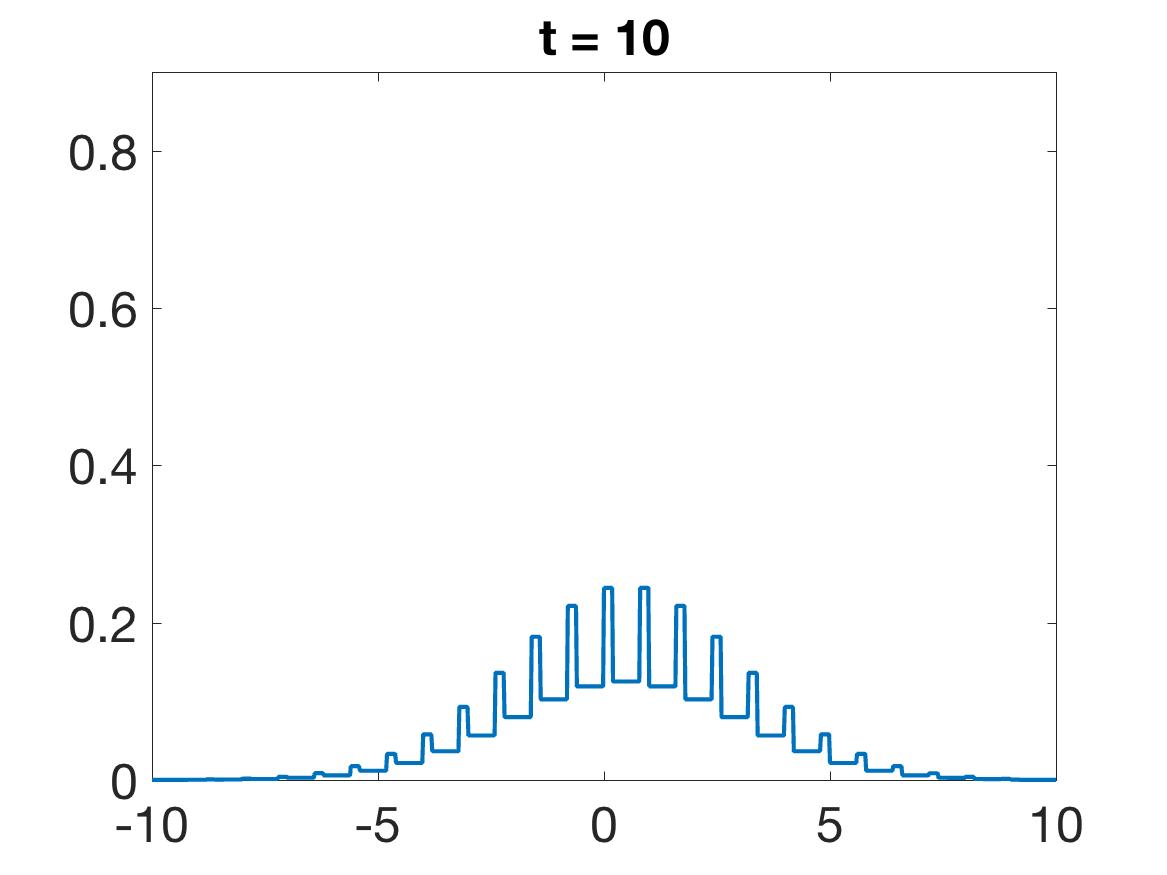}
\end{center}
\caption{Solution of the rescaled equation \eqref{eq:dynrescaled} with $\varepsilon=0.8$ and initial data $\fin\equiv \mathbbm{1}_{[0,1)}$.}
\label{fig2}
\end{figure}

\begin{figure}[h!]
\begin{center}
\includegraphics[width=0.32\textwidth]{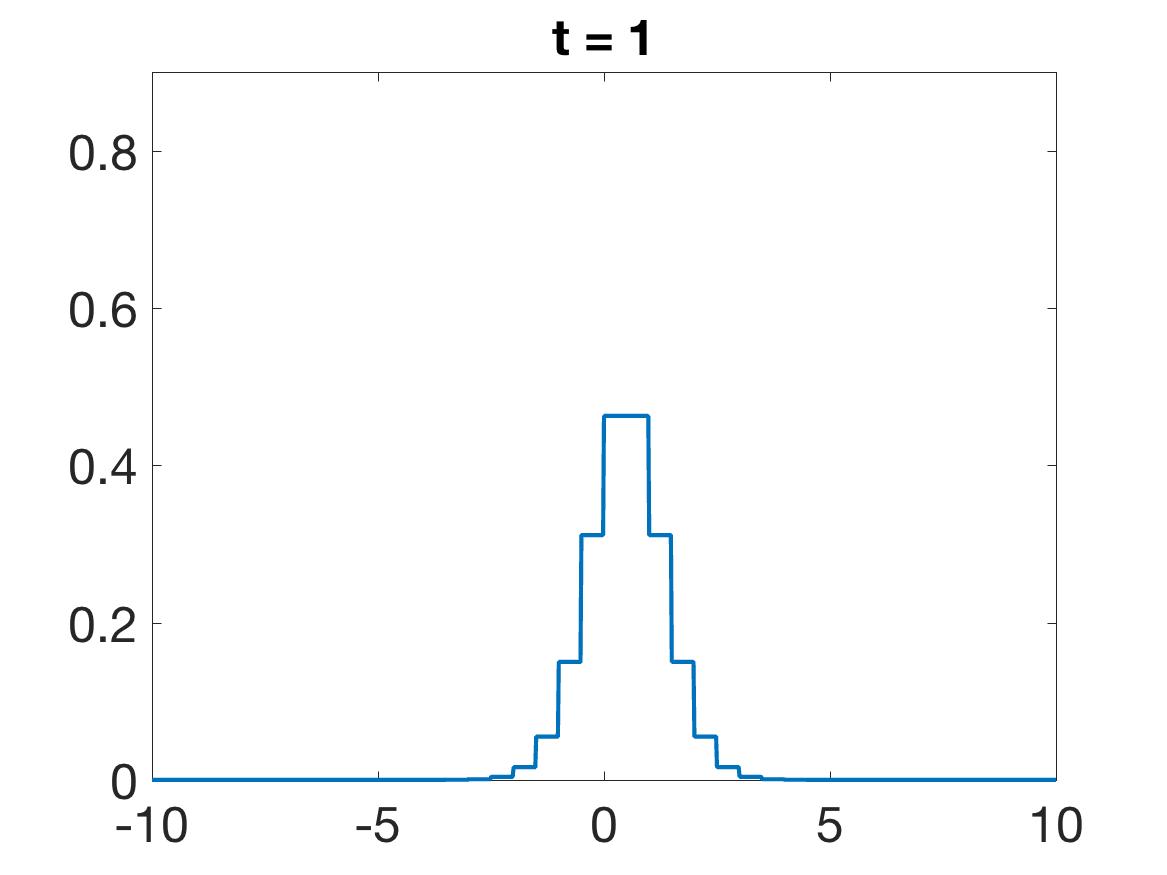}
\includegraphics[width=0.32\textwidth]{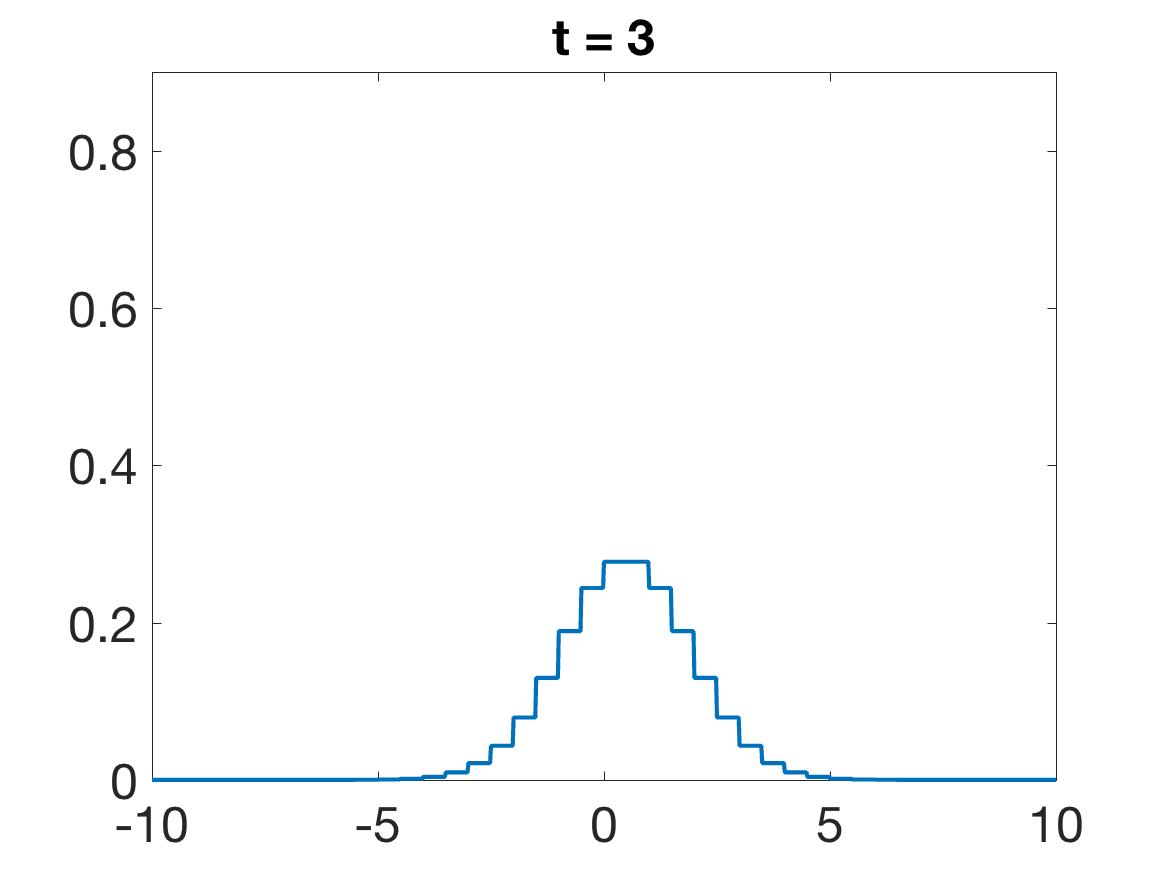}
\includegraphics[width=0.32\textwidth]{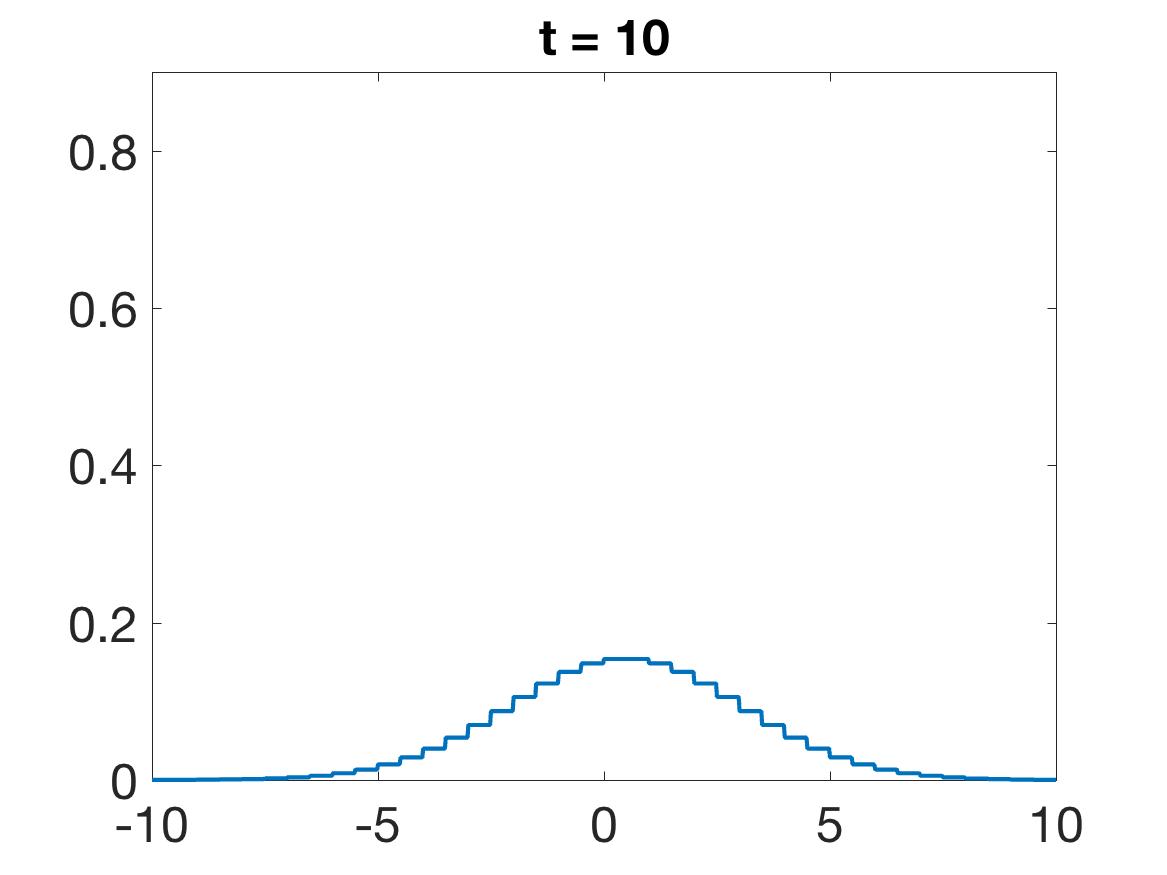}
\end{center}
\caption{Solution of the rescaled equation \eqref{eq:dynrescaled} with $\varepsilon=0.5$ and initial data $\fin\equiv \mathbbm{1}_{[0,1)}$.}
\label{fig3}
\end{figure}

We finally note that, when the payoff $h$ is greater than the magnitude of the support of the initial conditions, then there are some regions in which the density is zero for all $t\in\R^+$. This behavior shows that the solution may not have compact support even if the initial condition has compact support \textcolor{black}{(see Figures \ref{fig1}, \ref{fig2} and \ref{fig3})}.

\section{The constrained model}
\label{Sec4}

We now consider a framework in which the exchange variable cannot become negative, \ie $x\in\R^+$.
Then the model  takes the following form:
\begin{equation}\label{eq:dynnodebt}
\begin{cases}
\displaystyle \frac{1}{\eta}\dfdt = \frac{1}{3}\int_h^{+\infty}  f(t,x_*)\dx_*\left[\mathds{1}_{x\geq 2h}  f(t,x-h) 
+ \mathds{1}_{x\geq 0} f(t,x+h)
- \mathds{1}_{x\geq h} 2 f(t,x) \right]\\[13pt]
f(0,x) = \fin(x)
\end{cases}
\end{equation}
where we consider integrable and positive initial data, \textit{i.e.} $\fin\in L^1(\R^+)$ with $\fin(x)\geq 0$ for almost all $x\in\R^+$.

The model is clearly \textcolor{black}{nonlinear}. In what follows, we will work in the $L^1$ setting, which is particularly appropriate for this kind of problems because of the physical meaning of the density function $f$ and the typology of tools used in the proofs.

As for \eqref{eq:dyn}, we can easily show that the mass and the first moment are formally conserved: 
$$
\forall t\in [0,T], \quad \int_0^{+\infty} f(t,x) \dx = \Lnormp{\fin}:=\rho \quad \text{and} \int_0^{+\infty} xf(t,x) \dx = \Lnormp{x\fin}.
$$
However, this does not simplify equation \eqref{eq:dynnodebt}, which stays \textcolor{black}{nonlinear}. 
From here onward, the following notation will be useful:
$$
\beta(t)=\int_h^{+\infty}  f(t,x_*)\dx_*.
$$

\subsection{Basic properties}

We start by focusing on the behavior of the mass of the ``tail'' of the solution, $\beta(t)$.
We first provide a lower bound on $\beta(t)$. 
\begin{prop}\label{prop:betaconv}
The tail of the distribution $\beta$ satisfies
\begin{equation}\label{eq:beta1pos}
\beta(t)\geq \frac{\beta(0)}{\beta(0){\eta}t/3+1} \quad \text{ for all } t \in \R^+.
\end{equation}
\end{prop}

\begin{proof}

For all $k\in \N$, we define the quantities
$$
\beta_k(t) = \int_{kh}^{+\infty} f(t,x)\dx.
$$
In particular, $\beta_0\equiv \rho$ and $\beta_1\equiv\beta$.
We start by proving that $\beta_k(t)\geq 0$ for all $k\in\N$ and all $t\geq 0$.
Since $\fin\geq 0$, for all $k\in\N$, $\beta_k(0)\geq 0$. 
While for all $k\in\N$, $\beta_k(t)\geq 0$, 
we have
\begin{equation*}
\begin{cases}
\displaystyle \ddt \beta_1 = -\frac{\eta}{3}\beta_1(\beta_1-\beta_2)\geq -\frac{\eta}{3}\beta_1^2 \\[13pt]
\displaystyle \ddt \beta_k = \frac{\eta}{3}\beta_1(\beta_{k-1} + \beta_{k+1} - 2 \beta_k) \geq -2 \frac{\eta}{3}\beta_1\beta_k 
\end{cases}
\end{equation*}
which implies by integration that
\begin{equation}
\label{eq:betaconv}
\begin{cases}
\displaystyle  \beta_1(t) \geq \beta_1(0)\left( \beta_1(0)\frac{\eta}{3} t +1\right)^{-1} \geq 0 \\[13pt]
\displaystyle  \beta_k(t) \geq \beta_k(0) \exp\left (-\frac{2}{3}\eta\int_0^t\beta_1(s)\ds\right) \geq 0.
\end{cases}
\end{equation}
Hence $\beta_k(t)\geq 0$ for all \textcolor{black}{$t\in\R^+$ and \eqref{eq:beta1pos} holds}. 
\end{proof}

In particular, the previous proposition implies that if $\beta(0)\geq 0$, then $\beta(t)\geq 0 $ for all $t\in\R^+$. This allows us to prove the non-negativity of the solution to \eqref{eq:dynnodebt} starting from non-negative initial data.

\begin{prop}\label{prop:positive2}
Let $\fin\in L^1(\R^+)$ such that $\fin(x)\geq 0$ for almost every $x\in\R^+$. Then for almost all $(t,x)\in [0,T]\times \R^+$, $f(t,x)\geq 0$.
\end{prop}

\begin{proof}
Like in the proof of Lemma \ref{lemma:Linf},
we use the smoothing of the positive value function $\zeta$, and introduce the function 
$$
G:t\mapsto \int_{x\geq h} \int_0^{-f(t,x)} \zeta(s)\ds  \dx.
$$
Notice that the lower bound of the space integral constituting $G$ is different from that of the function defined in the proof of 
Lemma \ref{lemma:Linf}.
We compute the derivative of $G$:
\begin{equation*}
\begin{split}
G'(t) = & \textcolor{black}{-\int_h^{+\infty}} \zeta(-f(t,x))\partial_t f(t,x) \dx \\
 = & - \frac{\eta}{3}\beta(t) \int_h^{+\infty} \zeta(-f(t,x)) \left( f(t,x+h) + f(t,x-h)\mathds{1}_{x\geq 2 h} -2 f(t,x)\right) \dx \\
 = & - \frac{\eta}{3}\beta(t) \bigg[ \int_h^{+\infty} \zeta(-f(t,x)) ( f(t,x+h) - f(t,x) ) \dx - \int_{2h}^{+\infty} \zeta(-f(t,x)) ( f(t,x) - f(t,x-h) ) \dx  \\
& + \int_h^{2h} \zeta(-f(t,x))(-f(t,x))\dx \bigg] \\ 
 = & - \frac{\eta}{3}\beta(t) \bigg[ \int_{{3h}/{2}}^{+\infty} \zeta\left(-f\left(t,x-\frac{h}{2}\right)\right) \left( f\left(t,x+\frac{h}{2}\right) - f\left(t,x-\frac{h}{2}\right) \right) \dx  \\
& -\int_{{3h}/{2}}^{+\infty} \zeta\left(-f\left(t,x+\frac{h}{2}\right)\right) \left( f\left(t,x+\frac{h}{2}\right) - f\left(t,x-\frac{h}{2}\right) \right) \dx  
 + \int_h^{2h} \zeta(-f(t,x))(-f(t,x))\dx \bigg] \\ 
 = & - \frac{\eta}{3}\beta(t) \bigg\{ \int_{{3h}/{2}}^{+\infty} \left[\zeta\left(-f\left(t,x-\frac{h}{2}\right)\right)-\zeta\left(-f\left(t,x+\frac{h}{2}\right)\right) \right] \left[ 
 f\left(t,x+\frac{h}{2}\right) - f\left(t,x-\frac{h}{2}\right) \right] \dx  \\
& + \int_h^{2h} \zeta(-f(t,x))(-f(t,x))\dx \bigg\}. \\  
\end{split}
\end{equation*}
From the properties of $\zeta$, $(\zeta(x)-\zeta(y))(x-y)\geq 0$ for all $(x,y)\in\R^2$ and $x\, \zeta(x)\geq 0$ for all $x\in\R$. Furthermore, from Proposition \ref{prop:betaconv}, we know that $\beta(t)\geq 0$ for all $t \in [0,T]$.
Hence $G'(t)\geq 0$ for all $t\in [0,T]$. We deduce that $G(t)=0$ for almost all $t\in [0,T]$, which implies that $f(t,x)\geq 0$ for almost all $(t,x)\in [0,T]\times [h,+\infty)$.
Now referring to the dynamics \eqref{eq:dynnodebt}, for all $t\in [0,T]$ and $x\in [0,h)$: 
$$
\partial_t f(t,x) = \frac{\eta}{3}\beta(t) f(t,x+h) \geq 0
$$ 
since $f(t,x)\geq 0$ for almost all $x\geq h$.
Hence, we also have $f(t,x)\geq 0$ for almost all $(t,x)\in [0,T]\times [0,h)$.
\end{proof}

We now prove the well-posedness of \eqref{eq:dynnodebt}.
\begin{theorem}
Let $\fin\in L^1(\R^+)$ such that $\Vert \fin\Vert_{L^1(\R^+)}=\rho$, and $\fin\geq 0$. Let $T>0$.
The Cauchy problem \eqref{eq:dynnodebt} has a unique solution $f\in C([0,T];L^1(\R^+)) $.
\end{theorem}

\begin{proof}
We transform the Cauchy problem \eqref{eq:dynnodebt} in integral form.

Let
$$
A(f)(t,x) = \frac{\eta}{3}\int_0^t \int_h^\infty f(s,x_*)\dx_* \left(\mathds{1}_{x\geq 2 h}  f(s,x-h) 
+ \mathds{1}_{x\geq 0} f(s,x+h)
- \mathds{1}_{x\geq h} 2 f(s,x) \right) \ds
+\fin(x), 
$$
such that for all $t\in\R_+$, for all $x\in\R_+$, $f(t,x) = A(f)(t,x)$.
We denote by $F$ the set: 
$$
F = \{ f\in C([0,T];L^1(\R^+)) \text{ s.t. }\forall t\in [0,T], \Lnormp{f(t,\cdot)}=\rho, \; f\geq 0 \}.
$$
From Proposition \ref{prop:positive2}, it is clear that $A$ maps $F$ into $F$.

We now prove that $A$ is a contraction in $F$, for a suitable $T>0$. We define the operator $P(f)(t,x) = \mathds{1}_{x\geq 2 h}  f(t,x-h) 
+ \mathds{1}_{x\geq 0} f(t,x+h)
- \mathds{1}_{x\geq h} 2 f(t,x)$.
Let $(f_1, f_2)\in F\times F$.  
\begin{equation*}
\begin{split}
& \|A(f_1)-A(f_2)\|_{C([0,T]\times L^1(\R^+))} = \\
= &  \max\limits_{[0,T]} \frac{\eta}{3} \int_{\R^+} \bigg | \int_0^t \int_h^\infty f_1(s,x_*)\dx_* P(f_1)(s,x) \ds - \int_0^t \int_h^\infty f_2(s,x_*)\dx_* 
P(f_2)(s,x) \ds \bigg | \dx \\
 \leq & \frac{\eta}{3} \int_0^T\int_{\R^+} \bigg | \left(\int_h^\infty f_1(s,x_*)\dx_*\right) 
 P(f_1)(s,x) - \left( \int_h^\infty f_2(s,x_*)\dx_*\right) P(f_2)(s,x)  \bigg | \dx \ds \\
 \leq & \frac{\eta}{3} \int_0^T \int_h^\infty f_1(s,x_*)\dx_*
 \int_{\R^+} \left\vert P(f_1)(s,x)-P(f_2)(s,x)\right \vert \dx \\
 & +  \left\vert \int_h^\infty \left( f_1(s,x_*)-f_2(s,x_*)\right )\dx_* \int_{\R^+} P(f_2)(s,x)\dx  \right \vert \ds \\
\end{split}
\end{equation*}
Notice that
\begin{equation*}
\begin{split}
& \int_{\R^+} |P(f_1)(s,x)-P(f_2)(s,x)| \dx \leq \\
\leq & \int_{\R^+} | \mathds{1}_{x\geq 2 h}  (f_1-f_2)(s,x-h) |\dx 
+ \int_{\R^+} | \mathds{1}_{x\geq 0} (f_1-f_2)(s,x+h)|\dx
+ \int_{\R^+} | \mathds{1}_{x\geq h} 2( f_1-f_2)(s,x)|\dx \\
\leq & 4 \Lnormp{(f_1-f_2)(s,\cdot)}.
\end{split}
\end{equation*}
Furthermore, for $f\in F$, for all $s\in [0,T]$,
$$
\int_{\R^+} |P(f)(s,x)|\dx \leq 4\rho.
$$
Hence we obtain: 
\begin{equation*}
\begin{split}
\|A(f_1)-A(f_2)\|_{C([0,T]\times L^1(\R^+))} &\leq \frac{\eta}{3} \int_0^T (4 \rho \Lnormp{(f_1-f_2)(s,\cdot)}  +  4 \rho \Lnormp{(f_1-f_2)(s,\cdot)} ) \ds \\[10pt]
&\leq 8\frac{\eta}{3}\rho T \|f_1-f_2\|_{C([0,T]\times L^1(\R^+))}.
\end{split}
\end{equation*}
Therefore, for $T'<{3}/({8\eta\rho})$, Banach's fixed point theorem guarantees the existence of a unique fixed point, and hence a solution to the Cauchy problem \eqref{eq:dynnodebt} in $C([0,T']\times L^1(\R^+))$. A simple bootstrap argument guarantees the existence and uniqueness for all $T\geq 0$.
\end{proof}

The non-negativity of $f$ allows us to prove that $\beta(t)$ is non-increasing.
\begin{prop}\label{prop:beta1}
The function $t\mapsto\beta(t)$ satisfies:
$$
\ddt\beta(t)\leq 0.
$$
\end{prop}
\begin{proof}
Integrating \eqref{eq:dynnodebt} between $h$ and $\infty$, we get: 
$$
\ddt \beta(t) = \frac{\eta}{3}\beta(t) ( \beta(t)+\int_{2h}^\infty f(t,x)\dx  - 2\beta(t)) = - \frac{\eta}{3} \beta(t)  \int_h^{2h}f(t,x)\dx \leq 0.
$$
\end{proof}

We now show that $\beta(t)$ tends to zero at infinity, which amounts to saying that at infinity, all the mass $\rho$ gets concentrated on the interval $[0,h]$.
\begin{prop}\label{prop:beta2}
The function $t\mapsto\beta(t)$ satisfies 
$$
\lim_{t\rightarrow\infty} \beta(t)= 0.
$$
\end{prop}

\begin{proof}
Consider the functions $\beta_k$ satisfying \eqref{eq:betaconv}.
Since $f$ is differentiable with respect to time, $\beta_k\in C^1(\R^+)$ for all $k\in\N$.

By a bootstrap argument, we show that if for all $k\in\N$, $\beta_k\in C^N(\R^+)$, \textcolor{black}{$N\in\N$,} then from \eqref{eq:betaconv}, for all $k\in\N$, 
${\mathrm{d}}\beta_k/{\mathrm{d}t} \in C^N(\R^+)$ so $\beta_k\in C^{N+1}(\R^+)$. Hence $\beta_k\in C^\infty(\R^+)$ for all $k\in\N$.
Consider the sum
\begin{equation}\label{eq:sumbetak}
\Ddt\sum_{i=1}^k\beta_i = \frac{\eta}{3}\beta_1(\beta_{k+1}-\beta_{k}).
\end{equation}
We have 
\textcolor{black}{
$$
\sum_{i=1}^k\beta_i(t) = \sum_{i=1}^{k-1} i \int_{ih}^{(i+1)h} f(t,x)\dx + k\int_{kh}^{+\infty} f(t,x)\dx \leq \frac{1}{h}\sum_{i=1}^{k-1} \int_{ih}^{(i+1)h} x f(t,x)\dx +  \frac{1}{h} \int_{kh}^{+\infty} x f(t,x)\dx. 
$$
So for all $k\in\N$,
$$
\sum_{i=1}^k\beta_i(t) \leq \frac{1}{h} \int_{h}^{+\infty} x f(t,x) dx.
$$
}
Then 
$$
\lim_{k\rightarrow\infty}\sum_{i=1}^k\beta_i(t) \leq \frac{1}{h}\int_h^{\infty}xf(t,x)\dx \leq \frac{1}{h}\Lnormp{x\fin}.
$$ 
From Proposition \ref{prop:beta1}, since $\beta_1$ is \textcolor{black}{non-increasing} and bounded below by 0, there exists 
$$
c:=\lim_{t\rightarrow\infty}\beta_1(t).
$$
Suppose that $c>0$. Since $\beta_1\in C^2(\R^+)$, 
$$
\lim_{t\rightarrow\infty} \Ddt\beta_1= 0 = \lim_{t\rightarrow\infty}\left[-\frac{\eta}{3}\beta_1(t)(\beta_{1}(t)-\beta_{2}(t)) \right] = -\frac{\eta}{3} c (c-\lim_{t\rightarrow\infty}\beta_2(t)),
$$
hence 
$$
\lim_{t\rightarrow\infty}\beta_2(t)=c.
$$
Then from Equation \eqref{eq:sumbetak}, we deduce \textcolor{black}{by induction} that 
$$
\lim_{t\rightarrow\infty}\beta_k(t)=c
\text{ for all } k\in\N.
$$
This contradicts the fact that $\lim_{k\rightarrow\infty}\sum_{i=1}^k\beta_i(t)$ is finite, being bounded above by $\Lnormp{x\fin}/h$. Hence $c=0$, so 
$$
\lim_{t\rightarrow\infty} \beta(t)= 0.
$$
\end{proof}

In Proposition \ref{prop:beta1}, we showed that if $f$ is a solution to \eqref{eq:dynnodebt}, then $t\mapsto \beta(t) = \Lnormp{f(t,\cdot)\mathds{1}_{x\geq h}}$ is non-increasing. Additionally, we now show that $\Linormp{f(t,\cdot)\mathds{1}_{x\geq h}}$ is non-increasing. 

\begin{prop}\label{prop:linorm}
Let $\fin\in L^\infty(\R^+)$. Consider the solution $f$ to the Cauchy problem \eqref{eq:dynnodebt}, restricted to the interval $[h,+\infty)$, \textit{i.e.} $f\mathds{1}_{x\geq h}$.
Then $\|f \mathds{1}_{x\geq h}\|_{L^\infty([0,T]\times \R^+)} \leq \|\fin\mathds{1}_{x\geq h}\|_{L^\infty(\R^+)}$.
\end{prop}
\begin{proof}
We use again the function $\zeta$ defined in Lemma \ref{lemma:Linf} and define the function 

\textcolor{black}{
$$
G:t\mapsto \int_h^{+\infty} \int_0^{f(t,x)-K} \zeta(s)\ds \dx,
$$
}
where $K:=\Linorm{\fin\mathds{1}_{x\geq h} }$.
Then $G$ again satisfies: $G\in C^1((0,T],\R)$, $G(0)=0$ and $G(t)\geq 0$ for all $s\in [0,T]$. Differentiating it gives:
\begin{equation*}
\begin{split}
G'(t) =& \frac{\eta}{3}\beta(t) \int_h^{+\infty} \zeta(f(t,x)-K) ( f(t,x+h) + f(t,x-h)\mathds{1}_{x\geq 2h}  -2 f(t,x)) \dx \\
 = & \frac{\eta}{3}\beta(t) \bigg[ \int_h^{+\infty} \zeta(f(t,x)-K) ( f(t,x+h) - f(t,x) ) \dx \\
 & - \int_h^{+\infty} \zeta(f(t,x)-K) ( f(t,x) - f(t,x-h) ) \dx
  -\int_h^{2h} \zeta(f(t,x)-K) f(t,x-h) \dx \bigg] \\ 
 = & \frac{\eta}{3}\beta(t) \bigg[ \int_{{3h}/{2}}^{+\infty} \zeta (f(t,x-\frac{h}{2})-K ) \left( f(t,x+\frac{h}{2}) - 
 f(t,x-\frac{h}{2}) \right) \dx \\
 & - \int_{\frac{h}{2}}^{+\infty} \zeta (f(t,x+\frac{h}{2})-K ) \left( f(t,x+\frac{h}{2}) - f(t,x-\frac{h}{2}) \right) 
 \dx \\
& -\int_h^{2h} \zeta(f(t,x)-K) f(t,x-h) \dx \bigg] \\ 
  = & \frac{\eta}{3}\beta(t) \bigg\{ -\int_{\frac{3h}{2}}^{+\infty} \left[\zeta (f(t,x+\frac{h}{2})-K )-\zeta (f(t,x-\frac{h}{2})-K )\right]
  \left[ f(t,x+\frac{h}{2}) - f(t,x-\frac{h}{2}) \right] \dx \\
 & - \int_{\frac{h}{2}}^{\frac{3h}{2}} \zeta (f(t,x+\frac{h}{2})-K ) \left( f(t,x+\frac{h}{2}) - f(t,x-\frac{h}{2}) \right) \dx -
 \int_h^{2h} \zeta(f(t,x)-K) f(t,x-h) \dx \bigg\} .
\end{split}
\end{equation*}
The first term is negative since $\zeta$ is increasing function. After a change of variable, the sum of the last two integrals amounts to
$$
- \int_h^{2h} \zeta(f(t,x)-K) f(t,x) \dx 
$$
and again, due to the properties of $\zeta$, for all $s\in\R$, $\zeta(s-K)s\geq 0$. Hence $G'(t)\leq 0$ from which we deduce that $G(t) = 0$ for all $t\in [0,T]$, and the result follows.
\end{proof}

\subsection{The quasi-invariant limit}
\label{subsec:quasi-inv}

We now show that the limit when $h$ tends to zero of the solution to System \eqref{eq:dynnodebt}, in the diffusive scaling, satisfies a nonlinear
diffusion equation, with diffusion rate depending on the mass of the solution.

\begin{theorem}\label{th:existence}
Let $\fin\in L^1(\R^+) \cap L^\infty(\R^+)$ and let $T>0$.
We denote by $\fe \in C([0,T]; L^1(\R^+)) $ the unique solution to the rescaled Cauchy problem 
\begin{equation}\label{eq:systeps}
\begin{cases}
\displaystyle
\partial_t \fe(t,x) = \frac{\eta}{3\,\varepsilon^2}\left(\int_\varepsilon^{+\infty}  \fe(t,x_*)\dx_*\right)
\left[\mathds{1}_{x\geq 2\varepsilon}  \fe(t,x-\varepsilon) 
+ \mathds{1}_{x\geq 0} \fe(t,x+\varepsilon)
- \mathds{1}_{x\geq \varepsilon} 2 \fe(t,x) \right]\\[13pt]
\fe(0,x) = \fin(x).
\end{cases}
\end{equation}
Then there exists a subsequence $(\fe)_{\varepsilon>0}$ that converges weakly to 
$\tf \in \mathcal{M}([0,T]\times\R^+)$ and
$$
\tf(t,x) = f^+(t,x) + \left(\Lnormp{\fin} -\int_{\R^+} f^+(t,x) \dx \right) \delta_0(x) 
$$
where $f^+$ solves the following Cauchy problem:
\begin{equation}\label{eq:systlim}
\begin{cases}
\displaystyle \partial_t f^{\textcolor{black}{+}}(t,x) = \frac{\eta}{3}\left(\int_0^{+\infty}  f^{\textcolor{black}{+}}(t,x_*)\dx_*\right) \partial_x^2 
f^{\textcolor{black}{+}}(t,x) \\[10pt]
f^{\textcolor{black}{+}}(t,0) = 0 \quad & \textrm{ for a.e. } t\in \R^+ \\ 
f^{\textcolor{black}{+}}(0,x) = \fin(x) \quad &\textrm{ for a.e. } x\in \R^+.
\end{cases}
\end{equation}
\end{theorem}
The solution of equation \eqref{eq:systlim} is to be taken in the very weak sense, as defined below.

\begin{definition}
\label{def:veryweak}
A measurable function $f^{\textcolor{black}{+}}\in L^1([0,T]\times\R^+)$ is said to be a very weak solution of equation \eqref{eq:systlim} if it satisfies
\begin{equation}\label{eq:veryweaksol}
\int_0^T \!\! \int_{\R^+}\!\! f^{+}(t,x) \partial_t\varphi(t,x) \dx \dt + \frac{\eta}{3} \int_0^T \!\!
\int_{\R^+} \!\! f^{+}(t,x_*) \dx_* \int_{\R^+} \!\! 
f^{+}(t,x) \partial_x^2 \varphi(t,x) \dx \dt + \int_{\R^+} \!\!\fin(x)\varphi(0,x) \dx=0
\end{equation}
for all $\varphi\in C^1([0,T];C^2(\R^+))\cap L^\infty([0,T]\times\R^+)$, such that
$\varphi(T,x)=0$ for all $x\in\R^+$.
\end{definition}


\begin{proof}
For all $\varepsilon\in\R^+$, we denote by $\fe$ the unique solution to the Cauchy problem \eqref{eq:systeps}.
Let $\fep := \indp\fe$ and $\fem := \indm\fe$, so that $\fe = \fep + \fem$ and $\fep$ and $\fem$ have disjoint supports, respectively $(\varepsilon,+\infty)$ and $[0,\varepsilon]$.
It is easy to check that $\fep$ satisfies weakly  
\begin{equation}\label{eq:systepsp}
\begin{cases}
\displaystyle\partial_t \fep(t,x) = \frac{1}{\varepsilon^2}\frac{\eta}{3}\left( \int_\varepsilon^{+\infty}  \fep(t,x_*)\dx_*\right)
\left(\mathds{1}_{x\geq 2\varepsilon}  \fep(t,x-\varepsilon) 
+ \mathds{1}_{x\geq \varepsilon} \fep(t,x+\varepsilon) 
- \mathds{1}_{x\geq \varepsilon} 2 \fep(t,x) \right) \\[10pt]
\fep(0,x) = \fin(x)\indp.
\end{cases}
\end{equation}

Notice that
$$
\Lnormp{\fep(t,\cdot)} = \int_\varepsilon^{+\infty} \fe(t,x) \dx.
$$
From Proposition \ref{prop:beta1}, we know that $(\fep)_{\varepsilon>0}$ is bounded in $L^1([0,T]\times\R^+)$, with
$$
\int_0^T \int_0^\infty |\fep(t,x)| \dx \dt \leq T \Lnormp{\fin}.
$$
Furthermore, by Proposition \ref{prop:linorm}, $(\fep)_{\varepsilon>0}$ is also bounded in $L^\infty([0,T]\times\R^+)$ with 
$$
\Linormtp{\fep} \leq \Linormp{\fin}.
$$
This implies in particular that  $(\fep)_{\varepsilon>0}$ is equi-integrable. Hence by the Dunford-Pettis theorem,  $(\fep)_{\varepsilon>0}$ is relatively compact in $L^1([0,T]\times\R^+)$ with the weak topology.
It admits a weakly converging subsequence that we denote again by $(\fep)_{\varepsilon>0}$ such that  
\begin{equation}\label{eq:weakfept}
\forall \varphi\in L^\infty([0,T]\times\R^+), \quad \lim_{\varepsilon\rightarrow 0} \int_0^T \int_{\R^+} \fep(t,x) \varphi(t,x) \dx \dt = \int_0^T \int_{\R^+} \fp(t,x) \varphi(t,x) \dx \dt 
\end{equation}
with $\fp\in L^1([0,T]\times \R^+)$.

We can show in a similar way that for all $t\in [0,T]$, the sequence $(\fep(t,\cdot))_{\varepsilon>0}$ is bounded in $L^1(\R^+)$ and it is equi-integrable, so it admits a weakly converging subsequence that we denote still by $(\fep(t,\cdot))_{\varepsilon>0}$ such that 
\begin{equation}\label{eq:weakfep}
\forall \psi\in L^\infty(\R^+), \quad \lim_{\varepsilon\rightarrow 0} \int_{\R^+} \fep(t,x) \psi(x) \dx = \int_{\R^+} f_t^+(x) \psi(x) \dx 
\end{equation}
for $f_t^+\in L^1(\R^+)$.

We show that for all $t\in [0,T]$, $f_t^+ = \fp(t,\cdot)$. Let $t\in [0,T]$.  Let $\varphi\in L^\infty([0,T]\times\R^+)$ and consider the sequence 
$$ \left(\int_{\R^+} \fep(t,x) \varphi(t,x) \dx \right)_{\varepsilon>0}.
$$ 
From \eqref{eq:weakfep}, it converges to 
$$
\int_{\R^+} f_t^+(x) \varphi(t,x) \dx.
$$
Furthermore, 
$$
\left| \int_{\R^+} \fep(t,x) \varphi(t,x) \dx \right| \leq \Linormtp{\varphi} \Lnormp{\fin}.
$$
By dominated convergence, 
$$
\lim_{\varepsilon\rightarrow 0} \int_0^T \left( \int_{\R^+} \fep(t,x) \varphi(t,x) \dx \right) \dt
=  \int_0^T \lim_{\varepsilon\rightarrow 0}  \left( \int_{\R^+} \fep(t,x) \varphi(t,x) \dx \right) \dt.
$$
Hence from \eqref{eq:weakfept} and \eqref{eq:weakfep}, for all $t\in [0,T]$, $f_t^+ = \fp(t,\cdot)$.
In particular, taking $(t,x)\mapsto \varphi(t,x) \equiv 1$, we have the weak convergence of $\Lnormp{\fep}$, with 
\begin{equation}\label{eq:limfepl1}
\lim_{\varepsilon\rightarrow 0} \int_{\R^+} \fep(t,x) \dx = \int_{\R^+} \fp(t,x) \dx.
\end{equation}

Let us now show that $\fp$ satisfies weakly \eqref{eq:systlim}. Since $\fep$ satisfies weakly \eqref{eq:systepsp}, for all
test functions
$\varphi\in C^1([0,T];C^2(\R^+))\cap L^\infty([0,T]\times\R^+)$ such that for all $x\in\R^+$, $\varphi(T,x)=0$,
\begin{equation*}
\begin{split}
&-\int_0^T \int_{\R^+} \fep(t,x) \partial_t\varphi(t,x) \dx \dt - \int_{\R^+} \fin(x) \varphi(0,x) \dx \\
=  &\frac{\eta}{3}\frac{1}{\varepsilon^2}
 \int_0^T \left( \int_{\R^+} \fep(t,x_*) \dx_* \right) \int_{\R^+} \left(\mathds{1}_{x> 2\varepsilon}  \fep(t,x-\varepsilon) 
+ \mathds{1}_{x> \varepsilon} \fep(t,x+\varepsilon) 
- \mathds{1}_{x> \varepsilon} 2 \fep(t,x) \right) \varphi(t,x) \dx \dt. 
\end{split}
\end{equation*}
Notice that 
\begin{equation*}
\begin{split}
& \int_{\R^+} \left(\mathds{1}_{x> 2\varepsilon}  \fep(t,x-\varepsilon) 
+ \mathds{1}_{x> \varepsilon} \fep(t,x+\varepsilon) 
- \mathds{1}_{x> \varepsilon} 2 \fep(t,x) \right) \varphi(t,x) \dx  \\
= &  \int_{2\varepsilon}^{+\infty} \fep(t,x-\varepsilon) \varphi(t,x) \dx + \int_{\varepsilon}^{+\infty} \fep(t,x+\varepsilon) \varphi(t,x) \dx 
-2  \int_{\varepsilon}^{+\infty} \fep(t,x)\varphi(t,x) \dx  \\
= &  \int_{\varepsilon}^{+\infty} \fep(t,x) \varphi(t,x+\varepsilon) \dx + \int_{2\varepsilon}^{+\infty} \fep(t,x) \varphi(t,x-\varepsilon) \dx 
-2  \int_{\varepsilon}^{+\infty} \fep(t,x)\varphi(t,x) \dx  \\
= & \int_{\varepsilon}^{+\infty} \fep(t,x) \left( \varphi(t,x-\varepsilon) 
+ \varphi(t,x+\varepsilon) - 2 \varphi(t,x) \right) \dx -  \int_{\varepsilon}^{2\varepsilon} \fep(t,x)\varphi(t,x-\varepsilon) \dx \\
= & \ \varepsilon^2 \int_{\varepsilon}^{+\infty} \fep(t,x) \left( \partial_x^2 \varphi(t,x) + \mathcal{O}(\textcolor{black}{\varepsilon^2}) \right) \dx -  \int_{\varepsilon}^{2\varepsilon} \fep(t,x)\varphi(t,x-\varepsilon) \dx .
\end{split}
\end{equation*}
So $\fep$ satisfies, for all $\varphi\in C^1([0,T];C^2(\R^+))\cap L^\infty([0,T]\times\R^+)$ 
such that $\varphi(T,x)=0$ for all $x\in\R^+$:
\begin{equation}\label{eq:fep}
\begin{split}
&-\int_0^T \int_{\R^+} \fep(t,x) \partial_t\varphi(t,x) \dx \dt - \int_{\R^+} \fin(x) \varphi(0,x) \dx \\
=  &\frac{\eta}{3}
 \int_0^T \left( \int_{\R^+} \fep(t,x_*) \dx_* \right) \left( \int_{\varepsilon}^{+\infty} \fep(t,x)  \left( \partial_x^2 \varphi(t,x) +
 \mathcal{O}(\textcolor{black}{\varepsilon^2}) \right) \dx -  \int_{\varepsilon}^{2\varepsilon} \fep(t,x)\varphi(t,x-\varepsilon) \dx \right) \dt .
\end{split}
\end{equation}
We study the limit of each of these terms. 
Firstly, from \eqref{eq:weakfept}, since $\partial_t\varphi \in L^\infty([0,T]\times\R^+)$,
\begin{equation}\label{eq:limterm1}
\lim_{\varepsilon\rightarrow 0} \int_0^T \int_{\R^+} \fep(t,x) \partial_t\varphi(t,x) \dx \dt = \int_0^T \int_{\R^+} \fp(t,x) \partial_t\varphi(t,x) \dx \dt .
\end{equation}
Secondly, 
\begin{equation}\label{eq:limterm2}
\begin{split}
\left| \int_0^T \left( \int_{\R^+} \fep(t,x_*) \dx_* \right) \int_{\varepsilon}^{+\infty} \fep(t,x)  \mathcal{O}(\textcolor{black}{\varepsilon^2})\dx \dt \right| & = \mathcal{O}(\textcolor{black}{\varepsilon^2})\int_0^T \left( \int_{\R^+} \fep(t,x_*) \dx_* \right)^2 \dt \\
& \leq
\mathcal{O}(\textcolor{black}{\varepsilon^2}) \Vert{\fin}\Vert^2_{L^1(\R^+)} T \xrightarrow[\varepsilon \to 0]{} 0 .
\end{split}
\end{equation}
For the last term, we write: 
\begin{equation}\label{eq:limterm3}
\begin{split}
& \left| \int_0^T \left( \int_{\R^+} \fep(t,x_*) \dx_* \right)  \int_{\varepsilon}^{2\varepsilon} \fep(t,x)\varphi(t,x-\varepsilon) \dx \dt \right| \\
 \leq & \ \varepsilon  \int_0^T \textcolor{black}{\Lnormp{\fin}} \Linormp{\varphi(t,\cdot)}
\textcolor{black}{ \Linormp{\fin}}
\leq \varepsilon T \Lnormp{\fin} \Linormtp{\varphi} \textcolor{black}{ \Linormp{\fin}} \xrightarrow[\varepsilon \to 0]{} 0 .
\end{split}
\end{equation}
Lastly, we look at the \textcolor{black}{nonlinear} term.
From 
\textcolor{black}{\eqref{eq:limfepl1},}
we have 
\begin{equation}\label{eq:convbeta}
\int_{\R^+} \fep(t,x_*) \dx_* \xrightarrow[\varepsilon \to 0]{} \int_{\R^+} \fp(t,x_*) \dx_*
\end{equation}
\textcolor{black}{pointwise for all $t\in [0,T]$} and, from \textcolor{black}{\eqref{eq:weakfept}-\eqref{eq:weakfep}},
$$\int_{\R^+} \fep(t,x) \partial_x^2 \varphi(t,x) \dx \xrightarrow[\varepsilon \to 0]{} \int_{\R^+} \fp(t,x) \partial_x^2 \varphi(t,x) \dx$$
\textcolor{black}{pointwise} for all $t\in [0,T]$.
Since both of those limits are finite, for each $t\in [0,T]$, 
$$
 \int_{\R^+} \fep(t,x_*) \dx_* \left(\int_{\R^+} \fep(t,x_*) \partial_x^2 \varphi(t,x) \dx_* \right)  \xrightarrow[\varepsilon \to 0]{} \int_{\R^+} \fp(t,x_*) \dx_* \left(\int_{\R^+} \fp(t,x) \partial_x^2 \varphi(t,x) \dx\right).
$$
Furthermore, 
$$
\left| \int_{\R^+} \fep(t,x_*) \dx_* \int_{\varepsilon}^{+\infty} \fep(t,x)   \partial_x^2 \varphi(t,x) \dx \right|
\leq \Lnormp{\fin}^2 \Linormtp{\partial_x^2\varphi} .
$$
So by dominated convergence, 
\begin{equation}\label{eq:limterm4}
\begin{split}
\lim_{\varepsilon\rightarrow 0} &\int_0^T \int_{\R^+} \fep(t,x_*) \dx_* \int_{\varepsilon}^{+\infty} \fep(t,x)   \partial_x^2 \varphi(t,x) \dx \dt \\
&= \int_0^T \left(\int_{\R^+} \fp(t,x_*) \dx_* \int_{\R^+} \fp(t,x) \partial_x^2 \varphi(t,x) \dx\right) \dt.
\end{split}
\end{equation}
In conclusion, we put together \eqref{eq:limterm1}, \eqref{eq:limterm2}, \eqref{eq:limterm3} and \eqref{eq:limterm4} and taking the limit of the weak formulation \eqref{eq:fep} we obtain that, for all $\varphi\in C^1([0,T];C^2(\R^+))\cap L^\infty([0,T]\times\R^+)$ 
such that $\varphi(T,x)=0$ for all $x\in\R^+$, 
\begin{equation}\label{eq:fp}
\begin{split}
\int_0^T &\int_{\R^+} \fp(t,x) \partial_t\varphi(t,x) \dx \dt +\int_{\R^+} \fin(x) \varphi(0,x) \dx  \\
&+\frac{\eta}{3}
 \int_0^T \left(\int_{\R^+} \fp(t,x_*) \dx_* \int_{\R^+} \fp(t,x) \partial_x^2 \varphi(t,x) \dx \right) \dt=0.
 \end{split}
\end{equation}
Briefly, 
$\fp$ satisfies weakly: 
$$
\partial_t \fp(t,x) = \frac{\eta}{3} \left(\int_0^\infty \fp(t,x_*) \dx_*\right) \partial_x^2 \fp(t,x)
$$
with initial condition $\fp(t,x)=\fin(x)$.
Furthermore, for all $\varepsilon>0$, for all $t\in\R^+$, $\fe^+(t,0)=0$, which implies:
$$
\forall t\in\R^+, \quad \fp(t,0) = \lim_{\varepsilon\rightarrow 0} \fe^+(t,0) = 0.
$$
We now examine the limit of $\fem$. Firstly, $\supp(\fem) = [0,\varepsilon] \xrightarrow[\varepsilon\rightarrow 0]{} \{0\}$ \textcolor{black}{and
$\fem\geq 0$ for all $\varepsilon>0$.}
Furthermore, the supports of $\fep$ and $\fem$ are disjoint, so for each $t\in [0,T]$,
$$
\Lnormp{\fin} = \Lnormp{\fe(t,\cdot)} = \Lnormp{\fem(t,\cdot)} + \Lnormp{\fep(t,\cdot)}.
$$
From \eqref{eq:limfepl1}, $\lim_{\varepsilon\rightarrow 0}\Lnormp{\fep(t,\cdot)} = \Lnormp{\fp(t,\cdot)}$ exists, so
$$
\lim_{\varepsilon\rightarrow 0}\Lnormp{\fem(t,\cdot)} = \Lnormp{\fin} - \Lnormp{\fp(t,\cdot)}.
$$
Hence $\fem$ converges weakly in $\mathcal{M}((0,T)\times\R^+)$
and 
$$
\lim_{\varepsilon\rightarrow 0} \fem = (\Lnormp{\fin} - \Lnormp{\fp(t,\cdot)}) \delta_0(x).
$$

To conclude, $\fep$ converges weakly in $L^1([0,T]\times\R^+)$ and $\fem$ converges weakly in $\mathcal{M}((0,T)\times\R^+)$, so their sum $\fe = \fep + \fem$ also converges weakly to a limit $\tf = \fp + f^-$, where $\fp$ is a weak solution to \eqref{eq:systlim} and $f^- = (\Lnormp{\fin} - \Lnormp{\fp(t,\cdot)}) \delta_0(x)$.
\end{proof}

\subsection{Numerical simulations}\label{Sec:num_sim}

The simulations of the behavior of the constrained model exhibit very different dynamics with respect to the situation of the non-constrained model.
In all the simulations, we have used the same initial condition as in Subsection \ref{subs/numsim1}, i.e. $\fin\equiv \mathbbm{1}_{[0,1)}$.
The time evolution of the solution, plotted for various values of the payoff $\varepsilon$,
clearly shows the existence of a sub-population composed of individuals who are not allowed to play anymore, because the value of their exchange variable is less than $\varepsilon$.

\begin{figure}[h!]
\begin{center}
\includegraphics[width=0.32\textwidth]{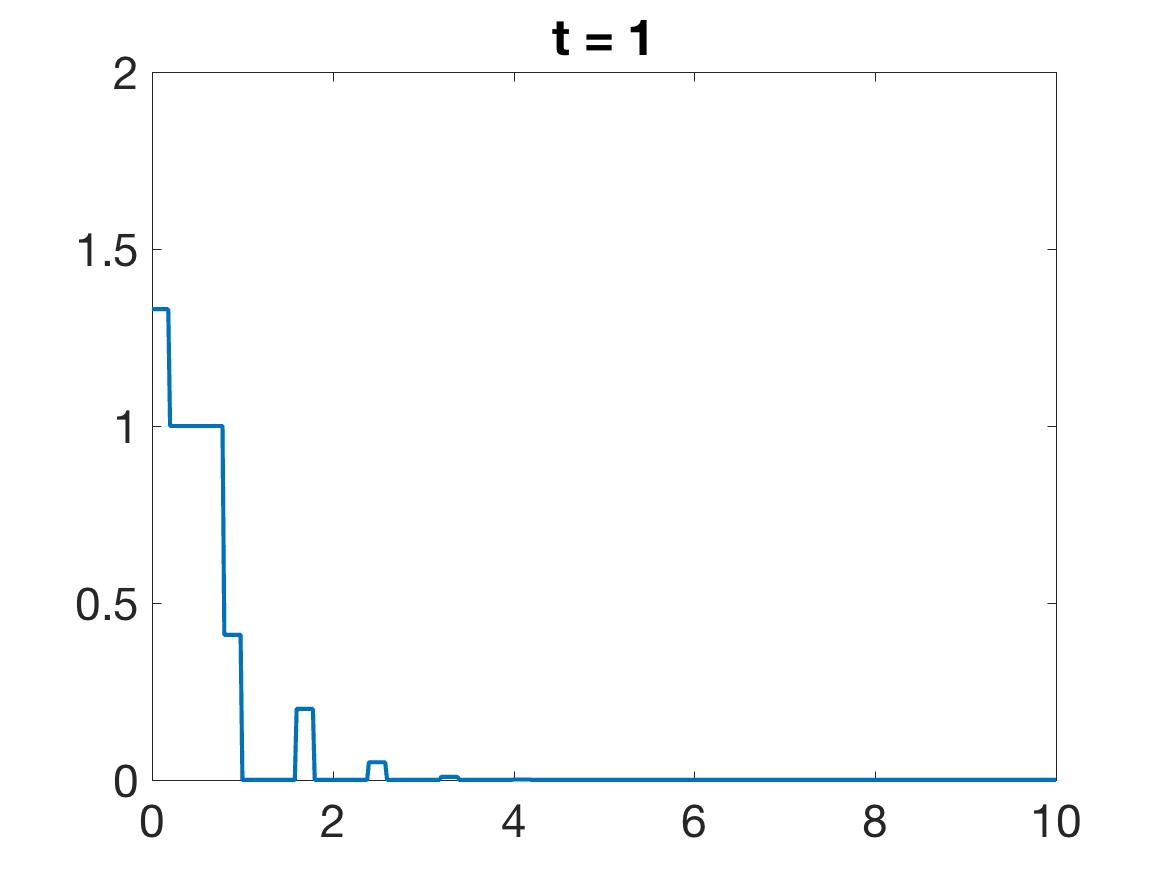}
\includegraphics[width=0.32\textwidth]{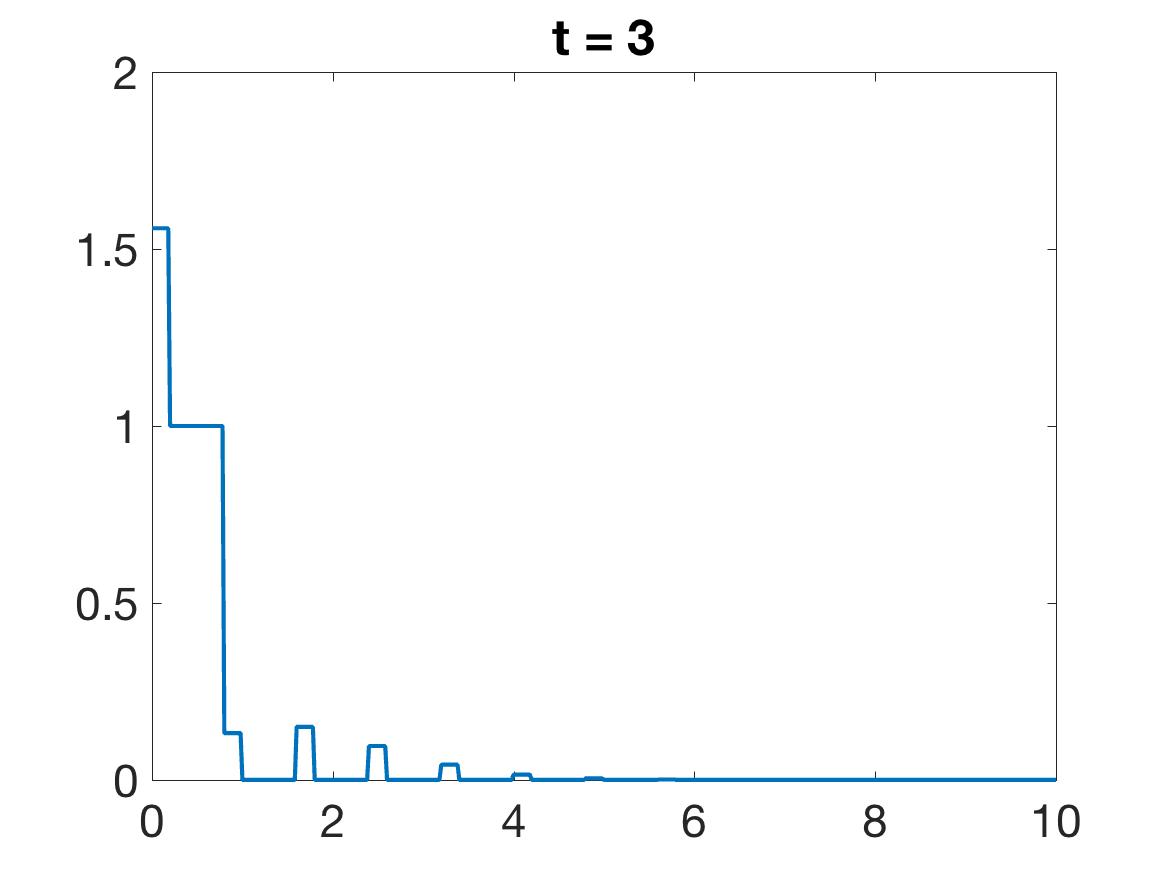}
\includegraphics[width=0.32\textwidth]{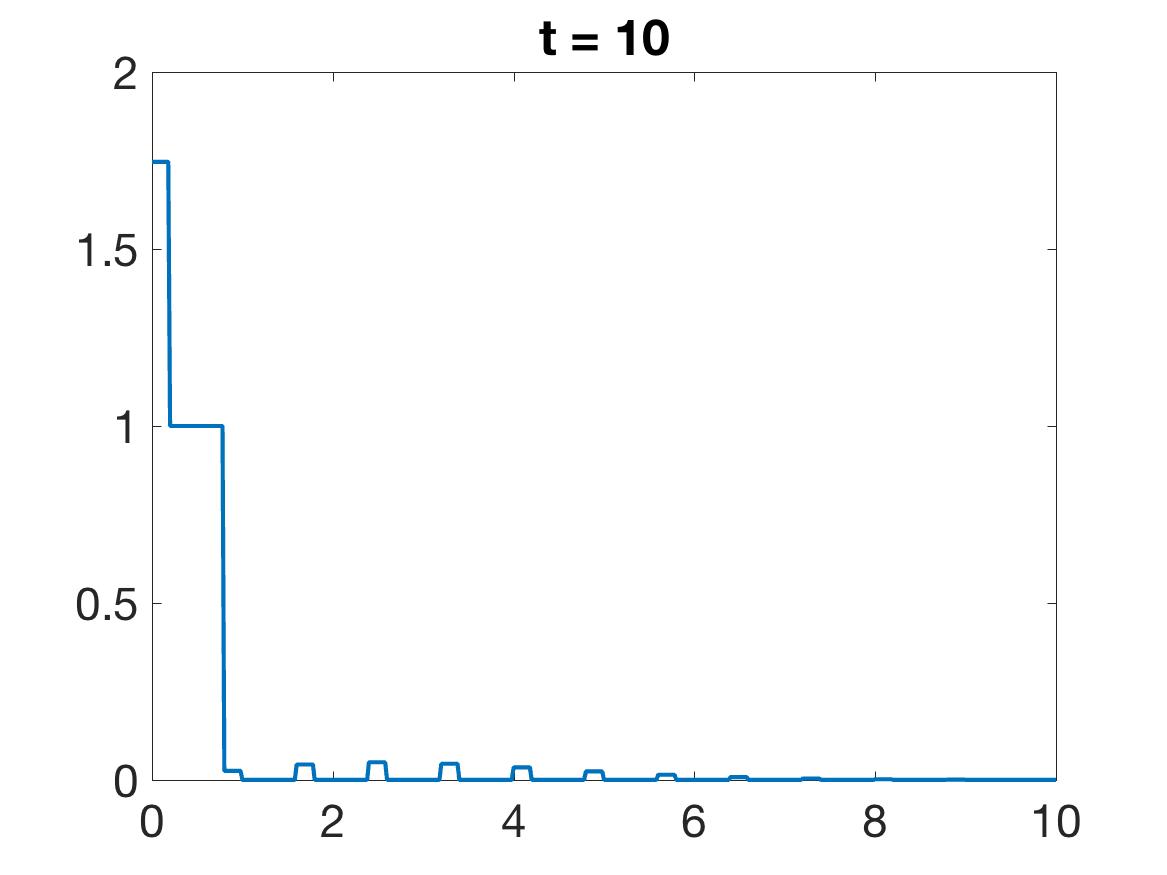}
\end{center}
\caption{Solution of the rescaled equation \eqref{eq:systeps} with $\varepsilon=0.8$ and initial data $\fin\equiv \mathbbm{1}_{[0,1)}$.}
\label{fig4}
\end{figure}

\begin{figure}[h!]
\begin{center}
\includegraphics[width=0.32\textwidth]{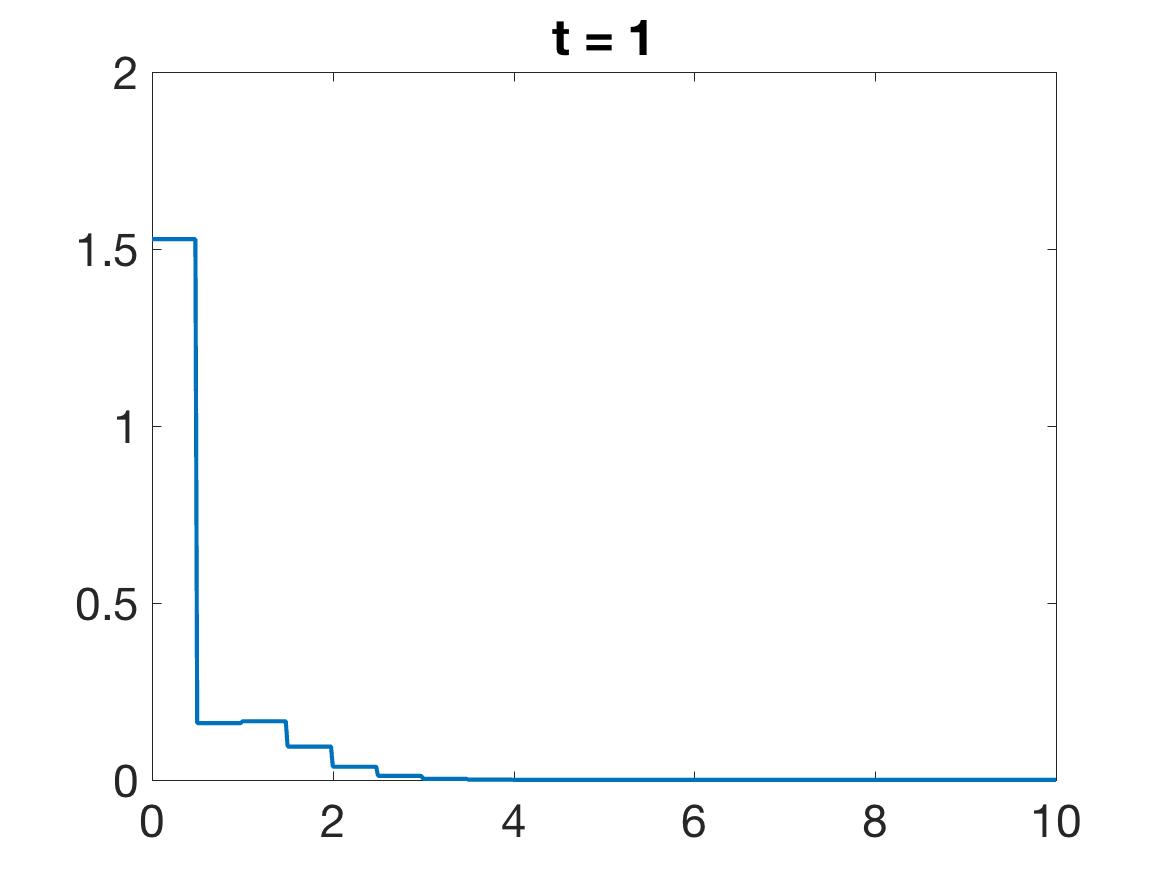}
\includegraphics[width=0.32\textwidth]{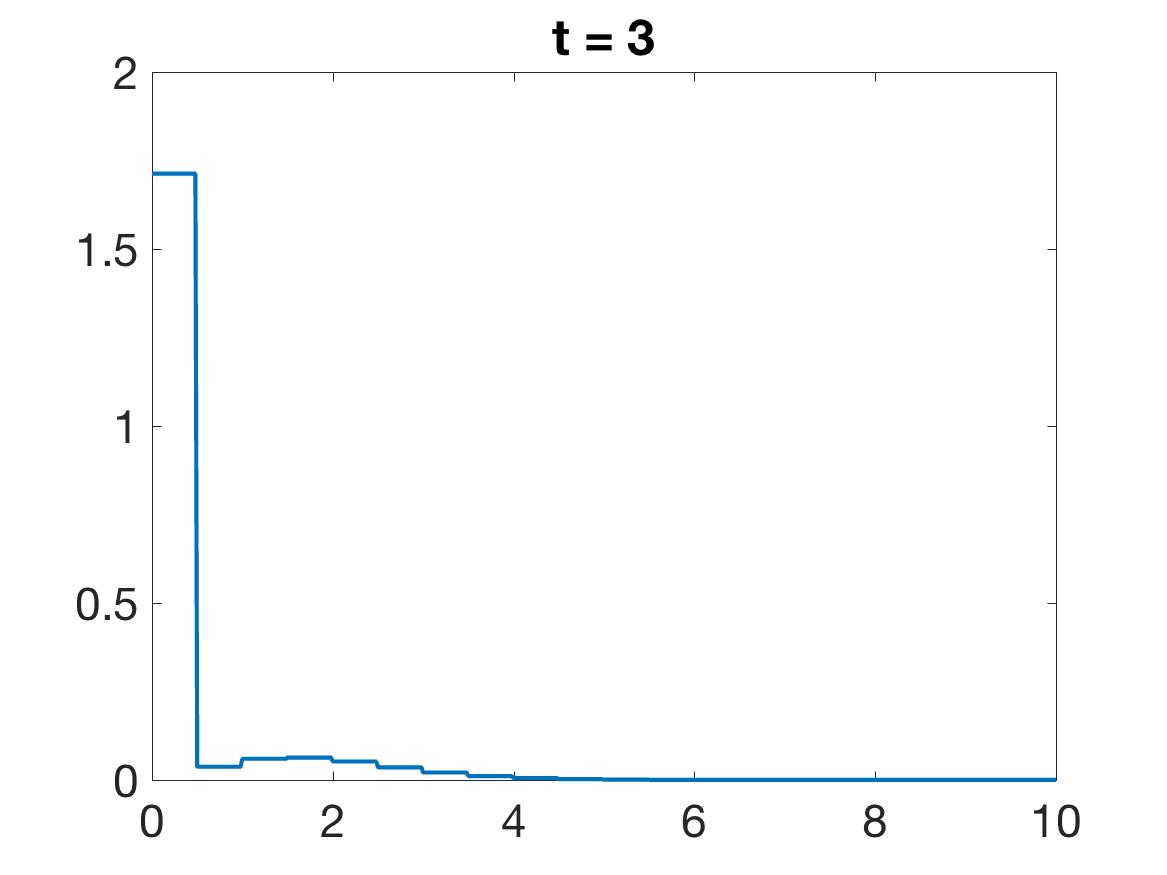}
\includegraphics[width=0.32\textwidth]{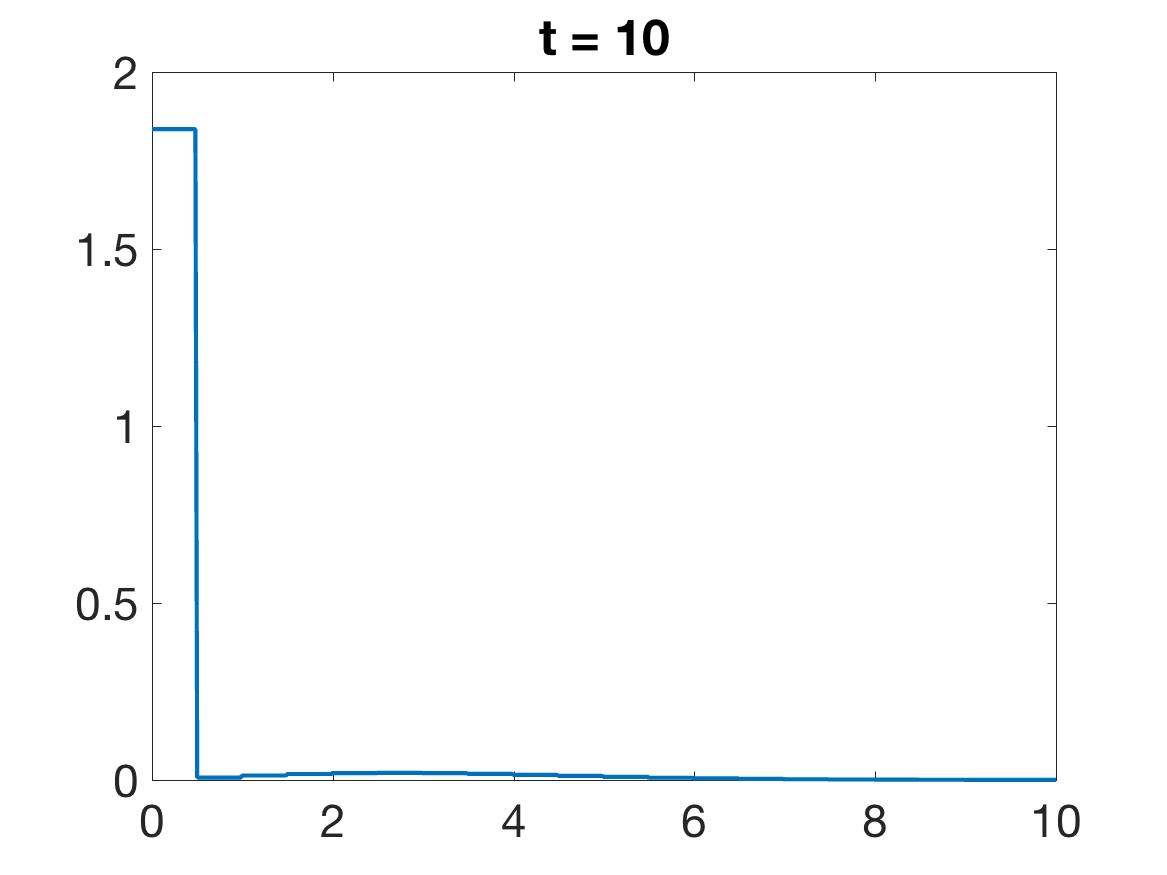}
\end{center}
\caption{Solution of the rescaled equation \eqref{eq:systeps} with $\varepsilon=0.5$ and initial data $\fin\equiv \mathbbm{1}_{[0,1)}$.}
\label{fig5}
\end{figure}

\begin{figure}[h!]
\begin{center}
\includegraphics[width=0.32\textwidth]{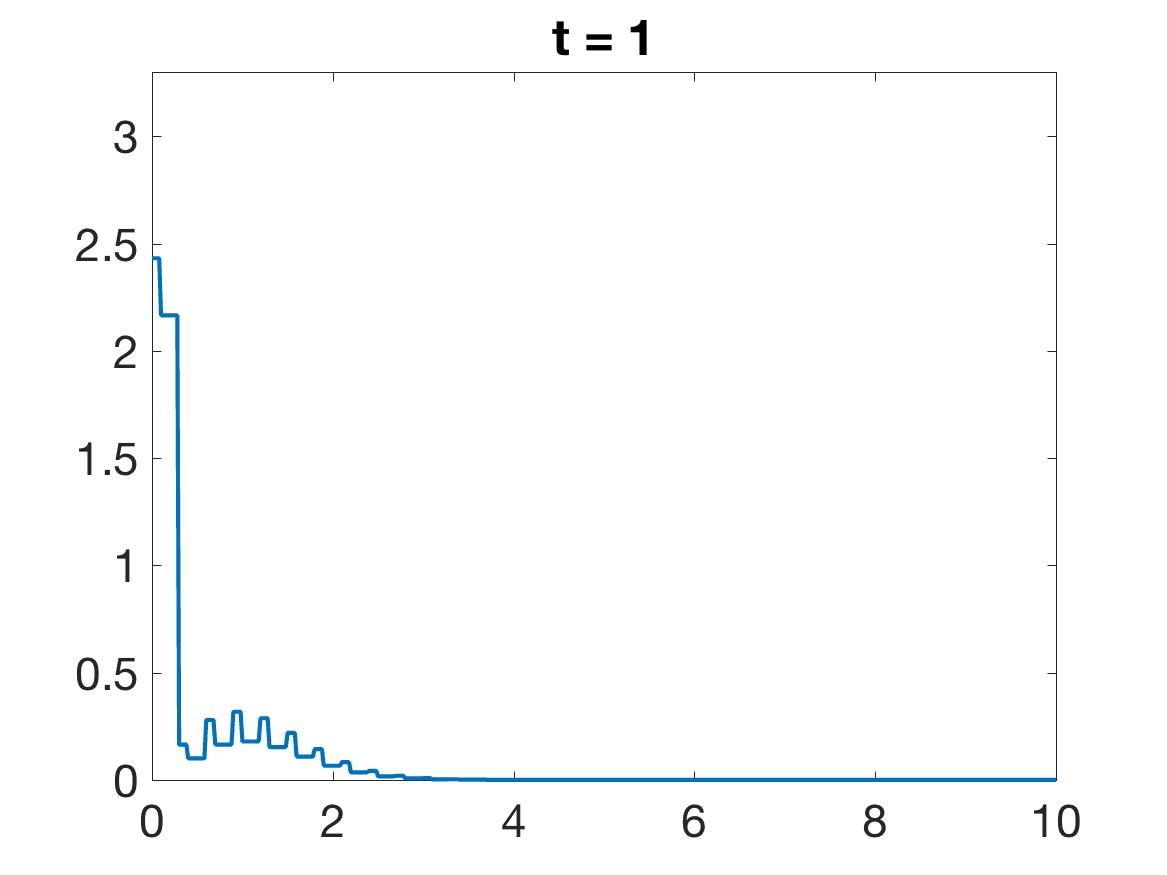}
\includegraphics[width=0.32\textwidth]{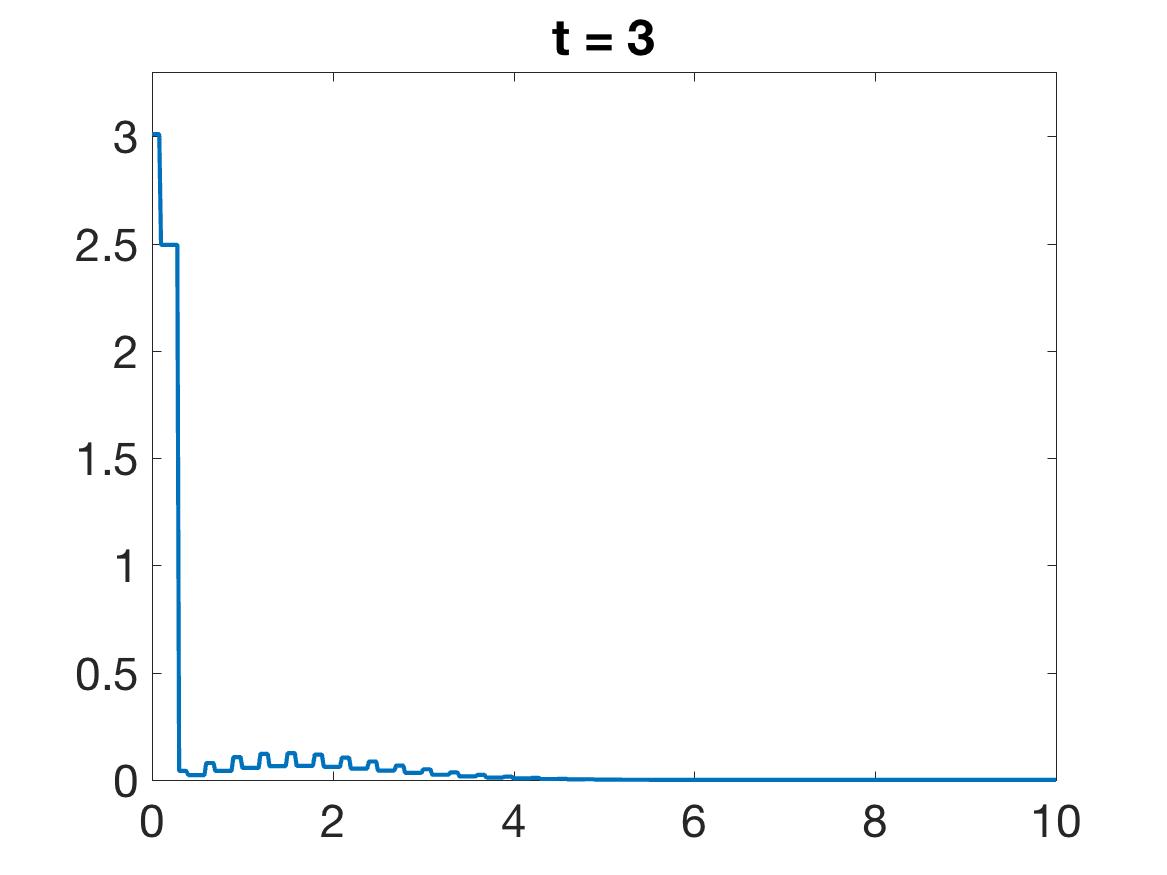}
\includegraphics[width=0.32\textwidth]{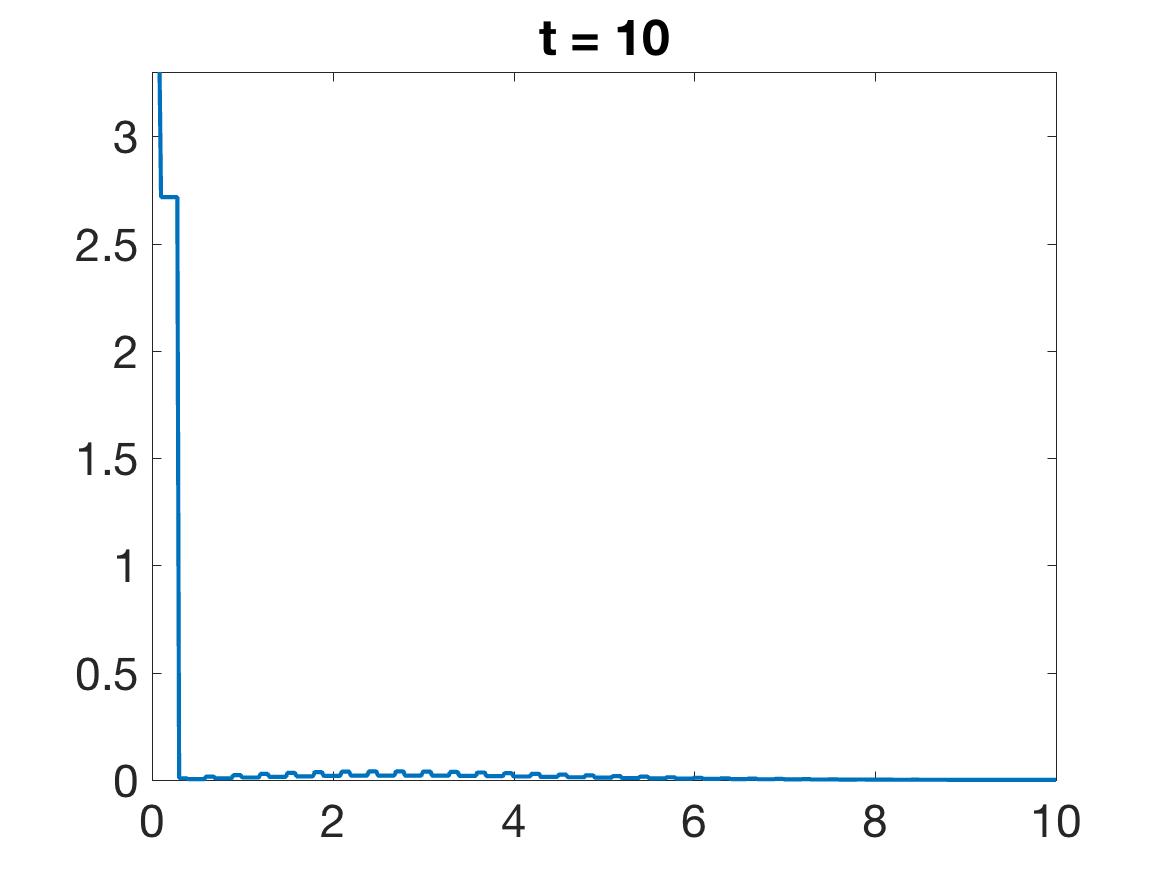}
\end{center}
\caption{Solution of the rescaled equation \eqref{eq:systeps} with $\varepsilon=0.3$ and initial data $\fin\equiv \mathbbm{1}_{[0,1)}$.}
\label{fig6}
\end{figure}

For small $\varepsilon$, the concentration effect around $x=0$ becomes more apparent \textcolor{black}{(see Figures \ref{fig4}, \ref{fig5}, \ref{fig6})}.

\section{The constrained continuous problem}
\label{Sec5}

In this section, we finally consider the continuous problem \eqref{eq:systlim} obtained in the limit $\varepsilon\rightarrow 0$.
\textcolor{black}{For the sake of simplicity, in this section we will denote the solution of \eqref{eq:systlim} simply by $f$, i.e.
we consider the problem:}
\begin{equation*}\label{eq:cont}
\begin{cases}
\displaystyle \partial_t f(t,x) = \frac{\eta}{3}\left( \int_{\R^+} f(t,x_*) \dx_*\right) \partial_x^2 f(t,x)\\[10pt]
f(t,0) = 0 \quad & \textrm{ for a.e. } t\in \R^+ \\ 
f(0,x) = \fin(x) \quad &\textrm{ for a.e. } x\in \R^+.
\end{cases}
\end{equation*}
Coherently with the original problem \eqref{eq:systeps}, we define 
$$
\bb = \int_{\R^+} f(t,x) \dx.
$$

The study of the well-posedness of the initial-boundary value problem \eqref{eq:systlim} can be easily based on the results of Subsection \ref{subsec:quasi-inv}.
Indeed, as an immediate corollary of Theorem \ref{th:existence} we see, by construction, that the solution of \eqref{eq:systlim}  exists and that,
moreover, it is a.e. non-negative if the initial condition satisfies $\fin\geq 0$ for a.e. $x\in\R^+$.

Therefore, in order to prove the well-posedness of the problem, we only need to show that the solution $f\textcolor{black}{\in L^1([0,T]\times\R^+)}$ to the limit problem \eqref{eq:systlim} is unique.
The following result holds.

\begin{theorem}\label{thm:wellposedness}
Let $\fin(x)\geq 0$ for all $x\in\R$. Then, there exists a unique non-negative very weak solution $f$ to \eqref{eq:systlim}.
\end{theorem}

\begin{proof}
Suppose, by contradiction, that there exist two very weak solutions, $f_1$ and $f_2$, to \eqref{eq:systlim}.

Then, consider the test function
$$
\tilde\varphi(t):=\int_t^T \psi(t')\dt'.
$$
It is clear that
$\tilde \varphi\in C^1([0,T];C^2(\R^+))\cap L^\infty([0,T]\times\R^+)$ and $\varphi(T,x)=0$ for all $x\in\R^+$
as soon as $\psi\in C^1([0,T])$.
Then, we can conclude, from the definition of very weak solution (Definition \ref{def:veryweak}), that
$$
\int_0^T\psi(t)\left(\int_{\R^+}(f_1-f_2)\dx\right)\dx \dt =0
$$
for all $\psi\in C^1([0,T])$,
i.e.
$$
\pi(t):=\int_{\R^+}f_1\dx =\int_{\R^+}f_2\dx.
$$

Clearly, $\pi$ is uniformly bounded, and hence belongs to $L^\infty(0,T)$.

Now, we come back to the definition of very weak solution. Suppose that there exist two very weak solutions, $f_1$ and $f_2$, and consider
\eqref{eq:veryweaksol}.

Then, we have that 
$$
\int_0^T \!\! \int_{\R^+}\!( f_1(t,x)-f_2(t,x)) \partial_t\varphi(t,x) \dx \dt + \frac{\eta}{3} \int_0^T 
\pi(t) \int_{\R^+} \!\! ( f_1(t,x)-f_2(t,x)) \partial_x^2 \varphi(t,x) \dx \dt =0
$$
for all $\varphi\in C^1([0,T];C^2(\R^+))\cap L^\infty([0,T]\times\R^+)$ and such that $\varphi(T,x)=0$ for all $x\in\R^+$.

Consider, for any $\psi^\star\in C^1_c([0,T];C^2(\R^+))\cap L^\infty([0,T]\times\R^+)$, the (backward) parabolic Cauchy problem
\begin{equation*}
\begin{cases}
\displaystyle \partial_t \varphi(t,x) + \frac{\eta}{3}\pi(t) \partial_x^2 \varphi(t,x)= \psi^\star  \quad & \textrm{ for a.e. }  (t,x)\in (0,T) \times\R^+  \\[4pt]
\varphi(t=T,x) = 0 \quad & \textrm{ for a.e. } x\in \R^+\\[4pt]
\varphi(t,x=0) = 0 \quad &\textrm{ for a.e. } t\in \R^+.
\end{cases}
\end{equation*}

By standard parabolic theory it is easy to prove that, for any $\psi^\star\in C^1_c([0,T];C^2(\R^+))\cap L^\infty([0,T]\times\R^+)$, there exists a uniformly bounded solution $\varphi\in C^1([0,T];C^2(\R^+))\cap L^\infty([0,T]\times\R^+)$ to the initial-boundary value problem.
Hence, 
$$
\int_0^T \!\! \int_{\R^+}\!( f_1(t,x)-f_2(t,x)) \psi^\star(t,x) \dx \dt =0
$$
for all $\psi^\star\in C^1([0,T];C^2(\R^+))\cap L^\infty([0,T]\times\R^+)$, which guarantees uniqueness of the solution.
\end{proof}

\bigskip
\noindent
{\bf Acknowledgments:}
Work supported by the ANR projects \textit{Kimega} (ANR-14-ACHN-0030-01), \textit{MFG} (ANR-16-CE40-0015-01) and by the Italian Ministry of Education, University and Research (\textit{Dipartimenti di Eccellenza} program 2018-2022, Dipartimento di Matematica 'F. Casorati', Universit\`a degli Studi di Pavia). \textcolor{black}{The authors thank moreover the University Paris-Dauphine, which hosted most of this research.}


\bibliography{biblio}

\begin{thebibliography}{10}

\bibitem{2012PhRvE}
C.~{Adami}, J.~{Schossau}, and A.~{Hintze}.
\newblock {Evolution and stability of altruist strategies in microbial games}.
\newblock {\em Phys. Rev. E}, 85(1):011914, 2012.

\bibitem{bou-sal2}
L.~Boudin and F.~Salvarani.
\newblock The quasi-invariant limit for a kinetic model of sociological
  collective behavior.
\newblock {\em Kinet. Relat. Models}, 2(3):433--449, 2009.

\bibitem{MR2744702}
L.~Boudin and F.~Salvarani.
\newblock Modelling opinion formation by means of kinetic equations.
\newblock In {\em Mathematical modeling of collective behavior in
  socio-economic and life sciences}, Model. Simul. Sci. Eng. Technol., pages
  245--270. Birkh\"auser Boston, Inc., Boston, MA, 2010.

\bibitem{bucur2016nonlocal}
C.~Bucur and E.~Valdinoci.
\newblock {\em Nonlocal diffusion and applications}, volume~1.
\newblock Springer, 2016.

\bibitem{golse2005hydrodynamic}
F.~Golse and L.~Saint-Raymond.
\newblock Hydrodynamic limits for the boltzmann equation.
\newblock {\em Riv. Mat. Univ. Parma (7)}, 4:1--144, 2005.

\bibitem{game2}
J.~Hofbauer, K.~Sigmund, et~al.
\newblock {\em Evolutionary games and population dynamics}.
\newblock Cambridge university press, 1998.

\bibitem{2011PhRvE}
L.-L. {Jiang}, T.~{Zhou}, M.~{Perc}, and B.-H. {Wang}.
\newblock {Effects of competition on pattern formation in the
  rock-paper-scissors game}.
\newblock {\em Phys. Rev. E}, 84(2):021912, 2011.

\bibitem{2002Natur}
B.~{Kerr}, M.~A. {Riley}, M.~W. {Feldman}, and B.~J.~M. {Bohannan}.
\newblock {Local dispersal promotes biodiversity in a real-life game of
  rock-paper-scissors}.
\newblock {\em Nature}, 418:171--174, 2002.

\bibitem{2004Natur}
B.~C. {Kirkup} and M.~A. {Riley}.
\newblock {Antibiotic-mediated antagonism leads to a bacterial game of
  rock-paper-scissors in vivo}.
\newblock {\em Nature}, 428:412--414, 2004.

\bibitem{mcnamara1992inelastic}
S.~McNamara and W.~Young.
\newblock Inelastic collapse and clumping in a one-dimensional granular medium.
\newblock {\em Physics of Fluids A: Fluid Dynamics}, 4(3):496--504, 1992.

\bibitem{Patk2012}
I.~P{\'a}tkov{\'a}, J.~J. {\v{C}}epl, T.~Rieger, A.~Blah{\r{u}}{\v{s}}kov{\'a},
  Z.~Neubauer, and A.~Marko{\v{s}}.
\newblock Developmental plasticity of bacterial colonies and consortia in
  germ-free and gnotobiotic settings.
\newblock {\em BMC Microbiology}, 12(1):178, 2012.

\bibitem{Shi2010BasinsOA}
H.~Shi, W.-X. Wang, R.~Yang, and Y.-C. Lai.
\newblock Basins of attraction for species extinction and coexistence in
  spatial rock-paper-scissors games.
\newblock {\em Physical review. E, Statistical, nonlinear, and soft matter
  physics}, 81 3 Pt 1:030901, 2010.

\bibitem{SL96}
L.~C.~M. Sinervo~B.
\newblock The rock--paper--scissors game and the evolution of alternative male
  strategies.
\newblock {\em Nature}, 380, 1996.

\bibitem{S64}
G.~Stampacchia.
\newblock \'{E}quations elliptiques du second ordre \`a coefficients
  discontinus.
\newblock {\em {S}\'eminaire Jean Leray}, (3):1--77, 1963-1964.

\bibitem{vazquez2007porous}
J.~L. V{\'a}zquez.
\newblock {\em The porous medium equation: mathematical theory}.
\newblock Oxford University Press, 2007.

\bibitem{game1}
J.~W. Weibull.
\newblock {\em Evolutionary game theory}.
\newblock MIT press, 1997.

\end{thebibliography}

\end{document}